\documentclass[11pt]{amsart}

\usepackage{relsize} 
\usepackage[bbgreekl]{mathbbol} 
\usepackage{amsfonts} 
\DeclareSymbolFontAlphabet{\mathbb}{AMSb} 
\DeclareSymbolFontAlphabet{\mathbbl}{bbold} 
\newcommand{\Prism}{{\mathlarger{\mathbbl{\Delta}}}}

\usepackage[T1]{fontenc}
\usepackage[latin1]{inputenc}
\usepackage{a4}
\usepackage{setspace}
\usepackage{xypic}
\usepackage{amssymb}
\usepackage{amsthm}
\usepackage[english]{babel}
\usepackage{latexsym}
\date{}
\newtheorem{thm}{Theorem}[section]
\newtheorem{lem}[thm]{Lemma}
\newtheorem{example}[thm]{Example}
\newtheorem{defn}[thm]{Definition}
\newtheorem*{thm*}{Theorem}
\newtheorem{rem}[thm]{Remark}
\newtheorem{conj}[thm]{Conjecture}

\newtheorem{prop}[thm]{Proposition}
\newtheorem{condition}[thm]{Condition}
\newtheorem{cor}[thm]{Corollary}

\newenvironment{f-proof}[1][\sc Proof of Corollary \ref{corCXnintro}.]{\begin{trivlist}
\item[\hskip \labelsep {\bfseries #1}]}{\hfill{$\square$}\end{trivlist}}

\newenvironment{ff-proof}[1][\sc Proof of Corollary \ref{corintro}.]{\begin{trivlist}
\item[\hskip \labelsep {\bfseries #1}]}{\hfill{$\square$}\end{trivlist}}

\newenvironment{fff-proof}[1][\sc Proof of Theorem \ref{thmpadic}.]{\begin{trivlist}
\item[\hskip \labelsep {\bfseries #1}]}{\hfill{$\square$}\end{trivlist}}

\newcommand{\Prod}{\displaystyle\prod}

\newcommand{\appl}[4]{
 \begin{array}{cccc}
  #1 & \longrightarrow & #2\\
  #3 & \longmapsto & #4
 \end{array}
}

\newcommand{\bq}{\mathbb Q}

\newcommand{\bz}{\mathbb Z}
\newcommand{\bn}{\mathbb N}

\newcommand{\br}{\mathbb R}
\newcommand{\bc}{{\mathbb C}}
\newcommand{\bt}{\mathbb T}
\newcommand{\bs}{\mathbb S}
\newcommand{\X}{{\mathcal X}}

\newcommand{\s}{{\mathrm{Spec}}}

\title{Topological Hochschild homology and Zeta-values}

\begin{document}

\author{Baptiste Morin}

\maketitle

\begin{abstract}
Using work of Antieau and Bhatt-Morrow-Scholze, we define a filtration on topological Hochschild homology and its variants $TP$ and $TC^-$ of quasi-lci rings with bounded torsion, which recovers the BMS-filtration after $p$-adic completion. Then we compute the graded pieces of this filtration in terms of Hodge completed derived de Rham cohomology relative to the base ring $\bz$. We denote the cofiber of the canonical map from $\mathrm{gr}^{n}TC^-(-)$ to $\mathrm{gr}^{n}TP(-)$ by $L\Omega^{<n}_{-/\bs}[2n]$. Let $\X$ be a regular connected scheme of dimension $d$ proper  over $\s(\bz)$ and let $n\in\bz$ be an arbitrary integer.  Together with Weil-étale cohomology with compact support $R\Gamma_{W,c}(\X,\bz(n))$, the complex $L\Omega^{<n}_{\X/\bs}$ is expected to give the Zeta-value $\pm\zeta^*(\X,n)$ on the nose. Combining the results proven here with a theorem recently proven in joint work with Flach, we obtain a formula relating $L\Omega^{<n}_{\X/\bs}$, $L\Omega^{<d-n}_{\X/\bs}$, Weil-étale cohomology of the archimedean fiber $\X_{\infty}$ with Tate twists $n$ and $d-n$, the Bloch conductor $A(\X)$ and the special values of the archimedean Euler factor of the Zeta-function $\zeta(\X,s)$ at $s=n$ and $s=d-n$. This formula is a shadow of the functional equation of Zeta-functions. 
\end{abstract}

\section{Introduction}

Let $\X$ be a regular scheme proper over $\mathrm{Spec}(\bz)$. The special value $\zeta^*(\X,n)$ of the zeta function $\zeta(\X,s)$ at an arbitrary integer argument $s=n\in\bz$ is conjecturally described in \cite{Flach-Morin18} in terms of two perfect complexes of abelian groups $R\Gamma_{W,c}(\X,\bz(n))$ and $R\Gamma(\X,L\Omega^{<n}_{\X/\bz})$ together with a canonical isomorphism
$$\lambda_{\infty}:\br\stackrel{\sim}{\rightarrow} \left(\mathrm{det}_{\bz}R\Gamma_{W,c}(\X,\bz(n))\otimes_{\bz}\mathrm{det}_{\bz}R\Gamma(\X,L\Omega^{<n}_{\X/\bz})\right)\otimes_{\bz}\br.$$
Here $R\Gamma(\X,L\Omega^{<n}_{\X/\bz})$ denotes derived de Rham cohomology modulo the $n$-th step of the Hodge filtration, as defined in \cite{Illusie72}.  
If  $\X$ is smooth over a number ring and $n\leq 1$, or if $\X$ lies over a finite field and $n\in\bz$ is arbitrary, then the equality
\begin{equation}\label{SVConj}
\lambda_{\infty}(\zeta^*(\X,n)^{-1})\cdot\bz=\mathrm{det}_{\bz}R\Gamma_{W,c}(\X,\bz(n))\otimes_{\bz}\mathrm{det}_{\bz}R\Gamma(\X,L\Omega^{<n}_{\X/\bz})
\end{equation}
follows from standard conjectures (see  \cite[Theorem 5.27]{Flach-Morin18}, \cite[Proposition 2.1]{Flach-Siebel19} and \cite{Morin16}). Note that (\ref{SVConj}) determines $\zeta^*(\X,n)$ up to sign. However, (\ref{SVConj}) is not quite true in general: the equality  (\ref{SVConj}) is expected to hold up to a certain correction factor $C(\X,n)$. For example, if $F$ is an abelian number field, $\X=\mathrm{Spec}(\mathcal{O}_F)$ and $n\geq 2$, then we have \cite[Proposition 5.34]{Flach-Morin18}
\begin{equation}\label{C(oF)}
C(\mathrm{Spec}(\mathcal{O}_F),n)=(n-1)!^{-[F:\bq]}.
\end{equation}
The motivation for this paper was to show that the equality (\ref{SVConj}) should hold for arbitrary $\X$ and arbitrary $n\in\bz$, provided one replaces in some sense the base ring $\bz$ by the sphere spectrum $\bs$ in the definition of  $R\Gamma(\X,L\Omega^{<n}_{\X/\bz})$. We realized that this philosophy might be true when we saw the computation of $\pi_*THH(\mathcal{O}_F)$ in \cite{Lindenstrauss-Madsen00}.

\subsection{The motivic filtration on topological Hochschild homology and its variants}

In order to define a complex which would play the role of $R\Gamma(\X,L\Omega^{<n}_{\X/\bs})$, recall from \cite{Antieau18} that negative cyclic homology $HC^-(\X)$ and periodic  cyclic homology $HP(\X)$ have natural filtrations $\mathrm{Fil}_B^*HC^-(\X)$ and $\mathrm{Fil}_B^*HP(\X)$ respectively, such that the cofiber of the canonical map
\begin{equation}\label{antieau}
\mathrm{gr}^n_B HC^-(\X)[-2n]\rightarrow \mathrm{gr}^n_B HP(\X)[-2n]
\end{equation}
is equivalent to $R\Gamma(\X,L\Omega^{<n}_{\X/\bz})$. Hence we may redefine
$$R\Gamma(\X,L\Omega^{<n}_{\X/\bz}):=\mathrm{Cofib}(\ref{antieau})=:\mathrm{gr}^n_B \Sigma^2HC(\X)[-2n].$$
Here, for any commutative ring $A$, we denote  Hochschild homology of $A/\bz$ by $$HH(A):=HH(A/\bz):=A^{\otimes_{\bz}\mathbb{T}}$$ 
where $\bt$ is the circle group, and we set $HC^-(A):=HH(A)^{h\bt}$, $HP(A):=HH(A)^{t\bt}$ and $HC(A):=HH(A)_{h\bt}$. 

Replacing $\bz$ by the sphere spectrum $\bs$, one obtains topological Hochschild homology $THH(A):=A^{\otimes_{\bs}\mathbb{T}}$, negative topological cyclic homology $TC^-(A/\bz):=THH(A)^{h\bt}$ and periodic  topological cyclic homology $TP(A):=THH(A)^{t\bt}$. We also consider "positive topological cyclic homology" $TC^+(A):=THH(A)_{h\bt}$, and we refer to \cite{Nikolaus-Scholze18} for these constructions (see also \cite{Hesselholt-Nikolaus19} for a survey). Therefore, we need to define a filtration on topological Hochschild homology and its variants, which is in some sense similar to the filtration $\mathrm{Fil}^*_B$ mentioned above. Recall from \cite{BMS} that such a filtration does exist,  after $p$-adic completion, for $p$-adically complete quasi-syntomic rings. We  denote by $\mathrm{Fil}_{BMS}^*$  the filtrations defined in \cite{BMS}.  We say that a commutative ring $A$ has bounded torsion if $A$ has bounded $p^{\infty}$-torsion for any prime $p$.  Finally, we denote by $(-)^{\wedge}_p$ the $p$-adic completion functor,  by $DF(R):=\mathrm{Fun}(\bz^{\mathrm{op}},D(R))$ the filtered derived $\infty$-category of some $\mathbb{E}_{\infty}$-ring $R$,  and by $\widehat{DF}(R)$ the full subcategory of $DF(R)$ spanned by the complete filtrations.

\begin{defn}\label{defintro} Let $A$ be a commutative ring with bounded torsion such that $L_{A/\bz}$ has Tor-amplitude in $[-1,0]$.  
We define  $F^*THH(A)$  by the pull-back square
\[ \xymatrix{
F^*THH(A)\ar[d]\ar[r]& \mathrm{Fil}_{HKR}^*HH(A) \ar[d]\\
\prod_p\mathrm{Fil}_{BMS}^*THH(A^{\wedge}_p,\bz_p)\ar[r]&\prod_p\mathrm{Fil}_{BMS}^*HH(A^{\wedge}_p/\bz_p,\bz_p) 
}
\]
of $\mathbb{E}_{\infty}$-algebra objects in the symmetric monoidal stable $\infty$-category $\widehat{DF}(\bs[\bt])$. Similarly, for $?=P,C^-$,
we define  $F^*T?(A)$  by the pull-back square
\[ \xymatrix{
F^*T?(A)\ar[d]\ar[r]& \mathrm{Fil}_{B}^*H?(A) \ar[d]\\
\prod_p\mathrm{Fil}_{BMS}^*T?(A^{\wedge}_p,\bz_p)\ar[r]&\prod_p\mathrm{Fil}_{BMS}^*H?(A^{\wedge}_p/\bz_p,\bz_p) 
}
\]
of $\mathbb{E}_{\infty}$-algebra objects in the symmetric monoidal stable $\infty$-category $\widehat{DF}(\bs)$.

There is a morphism $F^*TC^-(A)\rightarrow F^*TP(A)$ of $\mathbb{E}_{\infty}$-algebra objects, and we define $F^*\Sigma^2TC^+(A)\in \widehat{DF}(\bs)$ as the cofiber
$$F^*\Sigma^2TC^+(A):=\mathrm{Cofib}(F^*TC^-(A)\rightarrow F^*TP(A))$$
computed in the stable $\infty$-category $\widehat{DF}(\bs)$.
\end{defn}

The filtrations of Definition \ref{defintro} are functorial, multiplicative and complete by definition, and $F^*THH(A)$ is $\bt$-equivariant. We show that these filtrations are moreover exhaustive, and that $\mathrm{gr}_F^0TC^-(\bz)$ is an $\mathbb{E}_{\infty}$-$\bz$-algebra. It follows that the graded pieces of these filtrations are all $H\bz$-modules. Moreover, for any prime $p$, we have a canonical equivalence
$$(F^*T?(A))^{\wedge}_p\simeq \mathrm{Fil}_{BMS}^*T?(A^{\wedge}_p,\bz_p)$$
for $?=HH,P,C^-$. For any commutative ring $A$, we denote by $L_{A/\bz}$ the cotangent complex, by 
$L\widehat{\Omega}_{A/\bz}$ the Hodge completion of the derived de Rham complex \cite{Illusie72}, by  $L\widehat{\Omega}_{A/\bz}^{\geq n}$ the $n$-step of the Hodge filtration and by $L\Omega_{A/\bz}^{< n}$ the cofiber of the map $L\widehat{\Omega}_{A/\bz}^{\geq n}\rightarrow L\widehat{\Omega}_{A/\bz}$.  A bicomplete multiplicative bifiltration is a complete multiplicative filtration in $\widehat{DF}(R)$. The essential ingredients of the proof of the following result are the known computation \cite{Bökstedt85} of $\pi_*THH(\bz)$ and the fact \cite{BMS} that the key players are, after $p$-adic completion, concentrated in even degrees locally for the quasisyntomic topology.  

\begin{thm}\label{thm-intro} (cf. Section \ref{SectionBifiltglobal}). Let $A$ be a ring with bounded torsion such that $L_{A/\bz}$ has Tor-amplitude in $[-1,0]$. For $$?=HH,P,C^-,$$
there exists a functorial, $(\bn^{\mathrm{op}}\times\bz^{\mathrm{op}})$-indexed, multiplicative, bicomplete, biexhaustive bifiltration $\mathcal{Z}_{\bz}^*F^*T?(A)$ on $T?(A)$ such that
we have canonical equivalences
\begin{eqnarray*}
\mathcal{Z}_{\bz}^0F^*T?(A)&\simeq& F^*T?(A)\\
\mathrm{gr}^j_{\mathcal{Z}_{\bz}} \mathrm{gr}^{n}_{F}THH(A)&\simeq &L\Lambda^{n-j}L_{A/\bz}\otimes^L_{\bz} \bz/j[n+j-\epsilon_j]\\
\mathrm{gr}^j_{\mathcal{Z}_{\bz}} \mathrm{gr}^{n}_{F}TP(A) &\simeq & L\widehat{\Omega}_{A/\bz}\otimes^L_{\bz}\bz/j[2n-\epsilon_j]\\
\mathrm{gr}^j_{\mathcal{Z}_{\bz}} \mathrm{gr}^{n}_{F}TC^-(A)&\simeq & L\widehat{\Omega}^{\geq n-j}_{A/\bz}\otimes^L_{\bz}\bz/j[2n-\epsilon_j]\\
\mathrm{gr}^j_{\mathcal{Z}_{\bz}} \mathrm{gr}^{n}_{F}\Sigma^2TC^+(A)&\simeq & L\Omega^{< n-j}_{A/\bz}\otimes^L_{\bz}\bz/j[2n-\epsilon_j]
\end{eqnarray*}
for any $j\in\bn$ and any $n\in\bz$. Here $\epsilon_j:=\mathrm{Min}(1,j)$ and $\mathcal{Z}_{\bz}^*F^*\Sigma^2TC^+(A)$ is defined as the cofiber of the multiplicative map
$$\mathcal{Z}_{\bz}^*F^*TC^-(A)\rightarrow \mathcal{Z}_{\bz}^*F^*TP(A).$$
\end{thm}

In order to prove Theorem \ref{thm-intro}, we first show its $p$-adic version, in which case we allow more general base rings. We refer to \cite[Section 4]{BMS} for unexplained notation and terminology.

\begin{thm}\label{thmpadic}(cf. Section \ref{sectbifiltlocal}). Let $R\in \mathrm{QSyn}_{\bz_p}$ such that $(L\Lambda^{j}L_{R/\bz_p})^{\wedge}_p$ is a perfect complex of $R$-modules for all $j\geq0$, and let $A\in \mathrm{qSyn}_{R}$. If $R$ is perfectoid or $R=\bz_p$, we allow more generally $A\in \mathrm{QSyn}_{R}$. 

Then for $?=P,C^-$, there are $\bn^{\mathrm{op}}$-indexed, functorial, multiplicative, complete and exhaustive filtrations $\mathcal{Z}_{R}^*\mathrm{Fil}^*_{BMS} THH(A,\bz_p)$ and 
$\mathcal{Z}_{R}^*\mathrm{Fil}^*_{BMS} T?(A,\bz_p)$ on $\mathrm{Fil}^*_{BMS} THH(A,\bz_p)$ and $\mathrm{Fil}^*_{BMS} T?(-,\bz_p)$ respectively, endowed with canonical equivalences of filtrations
\begin{eqnarray*}
\mathrm{gr}^j_{\mathcal{Z}_{R}} \mathrm{Fil}^{*}_{BMS}THH(A,\bz_p)&\simeq & \mathrm{Fil}^{*-j}_{BMS}HH(A/R,\bz_p)\widehat{\otimes}_{R} \mathrm{gr}^{j}_{BMS}THH(R,\bz_p)\\
\mathrm{gr}^j_{\mathcal{Z}_{R}} \mathrm{Fil}^{*}_{BMS}T?(A,\bz_p)&\simeq & \mathrm{Fil}^{*-j}_{BMS}H?(A/R,\bz_p)\widehat{\otimes}_{R} \mathrm{gr}^{j}_{BMS}THH(R,\bz_p)
\end{eqnarray*}
for any $j\in\bn$, where the first equivalence is moreover $\bt$-equivariant. 
\end{thm}
The equivalences of Theorem \ref{thmpadic} are multiplicative in the obvious sense, and (seem to be) compatible with weights for the Adams operations on graded pieces, see \cite[Section 9.4]{BMS}.

Now we go back to global rings. We mention the following result since it gives what was expected in view of \cite{Lindenstrauss-Madsen00} and (\ref{C(oF)}).
\begin{cor}\label{corintro}
Let $F$ be a number field. Then we have
$$F^{*}THH(\mathcal{O}_F)\simeq \tau_{\geq 2*-1} THH(\mathcal{O}_F)$$
and
$$F^{*}\Sigma^2TC^+(\mathcal{O}_F)\simeq \tau_{\geq 2*-1} \Sigma^2TC^+(\mathcal{O}_F).$$
\end{cor}

It follows from Theorem \ref{thm-intro} that $\mathcal{Z}_{\bz}^jF^{n}\Sigma^2TC^+(A)\in \mathrm{Sp}_{\geq n+j}$ hence that the $DBF(\bs)$-valued functor $A\mapsto\mathcal{Z}_{\bz}^*F^*\Sigma^2TC^+(A)$ is left Kan extended from finitely generated polynomial $\bz$-algebras, see Corollary \ref{connectivity}, where $DBF(\bs)=\mathrm{Fun}(\bz^{\mathrm{op}}\times\bz^{\mathrm{op}}, \mathrm{Sp})$ denotes the $\infty$-category of bifiltered spectra. By left Kan extension we obtain a bicomplete biexhaustive bifiltration $\mathcal{Z}_{\bz}^*F^*\Sigma^2TC^+(A)$ for any commutative ring $A$, which moreover satisfies fpqc-descent, see Corollary \ref{corfiltS}.
\begin{defn}\label{defn/S}
We consider the fpqc-sheaf $$A\mapsto L\Omega_{A/\bs}^{<n}:=\mathrm{gr}^{n}_{F} \Sigma^2TC^+(A)[-2n]$$ and we define
$R\Gamma(\X,L\Omega_{\X/\bs}^{<n}):=R\Gamma(\X,L\Omega_{-/\bs}^{<n})$
by Zariski descent, for any scheme $\X$ and any $n\in\bz$.
\end{defn}

We have an equivalence
$$R\Gamma(\X,L\Omega_{\X/\bs}^{<n})\otimes^L_{\bz}\bq\stackrel{\sim}{\rightarrow} R\Gamma(\X_{\bq},L\Omega_{\X_{\bq}/\bq}^{<n}).$$ 
Moreover, for any prime number $p$, we have an equivalence
$$R\Gamma(\X,L\Omega_{\X/\bs}^{<n})^{\wedge}_p\simeq R\Gamma(\X^{\wedge}_p,\widehat{\Prism}_{\X^{\wedge}_p}\{n\}/\mathcal{N}^{\geq n})$$
where $\X^{\wedge}_p$ is the formal $p$-adic completion of $\X$, $\widehat{\Prism}_{\X^{\wedge}_p}\{n\}$ is the complex defined in \cite{BMS} and $\mathcal{N}^{\geq n}$ denotes its Nygaard filtration.

\subsection{The correcting factor and Zeta-values}
The Zeta-value conjecture formulated in \cite{Flach-Morin18} involves $R\Gamma_{W,c}(\X,\bz(n))$, $R\Gamma(\X,L\Omega^{<n}_{\X/\bz})$, the trivialization $\lambda_{\infty}$, and a correction factor $$C(\X,n):=\prod_{p<\infty} C_p(\X,n)\in\bq^{\times}$$ which is defined using $p$-adic Hodge theory. The corresponding Zeta-value conjecture \cite[Conjecture 5.12]{Flach-Morin18} is compatible with the Tamagawa number conjecture of Bloch-Kato and Fontaine-Perrin-Riou by \cite[Theorem 5.27]{Flach-Morin18}.  On the other hand, the Zeta-value conjecture formulated in \cite{Flach-Morin20} involves $R\Gamma_{W,c}(\X,\bz(n))$, $R\Gamma(\X,L\Omega^{<n}_{\X/\bz})$, the trivialization $\lambda_{\infty}$, and the following explicit correction factor 
$$C_{\infty}(\X,n):=\Prod_{i<n;j}(n-1-i)!^{(-1)^{i+j}\mathrm{dim}_{\bq}H^j(\X_{\bq},\Omega^{i})}$$
which we denote here by $C_{\infty}(\X,n)$ in order to distinguish it from $C(\X,n)^{-1}$. 
The corresponding Zeta-value conjecture \cite[Conjecture 1.1]{Flach-Morin20} is compatible with the functional equation of Zeta-functions by \cite[Theorem 1.3]{Flach-Morin20}. We expect that the rational numbers $C_{\infty}(\X,n)$ and $C(\X,n)^{-1}$ agree for arbitrary $\X$ and arbitrary $n$, but we do not address this question in the present paper, although we hope to return to it in future work.

Using Theorem \ref{thm-intro}, we obtain the following result.
\begin{cor}\label{corCXnintro}(cf. Section \ref{sectionC}).
Let $\mathcal{X}/\bz$ be a regular proper scheme over $\mathrm{Spec}(\bz)$.  Then $R\Gamma(\X,L\Omega_{\X/\bs}^{<n})$ is a perfect complex of abelian groups given with a fiber sequence
$$R\Gamma(\X, \mathcal{Z}^1L\Omega_{\X/\bs}^{<n})\rightarrow R\Gamma(\X,L\Omega_{\X/\bs}^{<n})\rightarrow R\Gamma(\X,L\Omega_{\X/\bz}^{<n})$$
where $R\Gamma(\X, \mathcal{Z}^1L\Omega_{\X/\bs}^{<n})$ has finite cohomology groups such that
$$\Prod_{i\in\bz}\mid H^{i}(\X,\mathcal{Z}^1L\Omega_{\X/\bs}^{<n})\mid^{(-1)^{i}}=C_{\infty}(\X,n)^{-1}.$$
\end{cor}

Corollary \ref{corCXnintro} and \cite{Morin16} immediately give the following
\begin{cor}\label{Milne}
Let $\X$ be a smooth proper scheme over the finite field $\mathbb{F}_q$. Then $R\Gamma(\X,L\Omega_{\X/\bs}^{<n})$ is a perfect complex with finite cohomology groups and we have 
$$\Prod_{i\in\bz}\mid H^{i}(\X,L\Omega_{\X/\bs}^{<n})\mid^{(-1)^{i}}=q^{\chi(\X/\mathbb{F}_q,\mathcal{O}_{\X},n)}$$ 
where the right hand side is Milne's correcting factor \cite{Milne86}.
\end{cor}
We now recall some constructions from \cite{Flach-Morin18}. Let $\X$ be a regular proper scheme over $\mathrm{Spec}(\bz)$. We denote by $\mathcal{X}_{\infty}$ the quotient topological space $\X(\bc)/G_\br$, where $\X(\bc)$ is endowed with the complex topology and $G_{\br}$ is the Galois group of $\bc/\br$. The space $\X_{\infty}$ is the fiber of the Artin-Verdier compactification $\overline{\X}_{et}$ of the étale topos $\X_{et}$ over the archimedean prime
$\infty\in\overline{\mathrm{Spec}(\bz)}_{et}$. We consider the morphism of topoi
$$\pi:\mathrm{Sh}(G_{\br},\X(\bc))\rightarrow \mathrm{Sh}(\mathcal{X}_{\infty})$$
where $\mathrm{Sh}(G_{\br},\X(\bc))$ is the topos of $G_{\br}$-equivariant sheaves on $\X(\bc)$ and  $\mathrm{Sh}(\X_{\infty})$ is the topos of sheaves on $\X_{\infty}$.
For any $n\geq 0$,  we define
$$i_{\infty}^*\mathbb{Z}(n):=\tau^{\leq n}R\pi_*(2i\pi)^n\bz$$
where $(2i\pi)^n\bz$ is seen as a $G_{\br}$-equivariant sheaf on $\X(\bc)$. Note that, if $a>0$ then $R^{a}\pi_*(2i\pi)^n\bz$ is a sheaf supported on the closed subspace $\X(\br)\subseteq\X_{\infty}$ with stalks $H^{a}(G_{\br}, (2i\pi)^n\bz)$. For $n<0$, we consider the continuous map of topological spaces 
$$f:(\mathbb{A}^{-n}_{\X})_{\infty}:=\mathbb{A}^{-n}_{\X}(\bc)/G_{\br}\rightarrow \X_{\infty}$$
and we set $$i_{\infty}^*\mathbb{Z}(n):=Rf_!\bz[-2n].$$
Then we define the perfect complex of abelian groups
$$R\Gamma_{W}(\mathcal{X}_{\infty},\mathbb{Z}(n)):=R\Gamma(\mathcal{X}_{\infty},i_{\infty}^*\mathbb{Z}(n)) \textrm{ for any }n\in\bz$$
where the right hand side is the hypercohomology of the topological space $\mathcal{X}_{\infty}$ with coefficients in the complex of sheaves $i_{\infty}^*\mathbb{Z}(n)$ defined above.
For any $n\in \bz$, we have an equivalence
\begin{equation}\label{betti}
R\Gamma_{W}(\mathcal{X}_{\infty},\mathbb{Z}(n))\otimes_{\bz}\br\simeq R\Gamma(\X(\bc),(2i\pi)^n\br)^{G_{\br}}.
\end{equation}
Assume moreover that $\X$ is connected of dimension $d$. We define the invertible $\bz$-module 
\begin{eqnarray*}
\Xi(\mathcal{X}/\mathbb{S},n)&:=&\mathrm{det}_{\mathbb{Z}}R\Gamma_{W}(\mathcal{X}_{\infty},\mathbb{Z}(n))\otimes_{\bz}\mathrm{det}^{-1}_{\mathbb{Z}} R\Gamma(\X,L\Omega_{\X/\bs}^{<n}) \\
& &\otimes_{\bz}  \mathrm{det}^{-1}_{\mathbb{Z}}R\Gamma_{W}(\mathcal{X}_{\infty},\mathbb{Z}(d-n))
\otimes_{\bz}\mathrm{det}_{\mathbb{Z}}R\Gamma(\X,L\Omega_{\X/\bs}^{<d-n}).
\end{eqnarray*}
It follows from (\ref{betti}) and from Corollary \ref{corCXnintro} that duality for Deligne cohomology  \cite[Lemma 2.3(b)]{Flach-Morin18} yields a canonical isomorphism 
$$\xi_{\mathcal{X}/\mathbb{S},n}:\mathbb{R}\stackrel{\sim}{\longrightarrow} \Xi(\mathcal{X}/\mathbb{S},n)\otimes\mathbb{R}$$
and we denote by $\mathrm{det}(\xi_{\mathcal{X}/\mathbb{S},n})$ its determinant computed with respect to the basis $1\in\br$ and a generator of $\Xi(\mathcal{X}/\mathbb{S},n)$, so that $\mathrm{det}(\xi_{\mathcal{X}/\mathbb{S},n})\in \br^{\times}/\{\pm 1\}$ is well defined up to sign.

We denote by $\zeta^*(\mathcal{X}_\infty,n)$ the leading Taylor coefficient at $s=n$ of the archimedean Euler factor of $\X$, see \cite[Section 4]{Flach-Morin20}.
Finally, we denote by $A(\X)$ the Bloch conductor (see \cite{Bloch87} and \cite[Definition 3.2]{Flach-Morin20}) if $\X/\bz$ is flat, and we set $A(\X):=q^{-\chi(\X/\mathbb{F}_q)}$ if $\X$ lies over the finite field $\mathbb{F}_q$, where $\chi(\X/\mathbb{F}_q)$ is defined in \cite[Section 10]{Milne86} in various ways.

The following result follows immediately from Corollay \ref{corCXnintro} and \cite[Theorem 1.2]{Flach-Morin20} (respectively from Corollary \ref{Milne} and \cite[Lemma 10.1]{Milne86}) if $\X/\bz$ is flat (respectively if $\X$ lies over a finite field).
\begin{thm}\label{thmintro} (Joint with Flach \cite{Flach-Morin20})
Let $\X$ be a regular connected scheme of dimension $d$, proper over $\mathrm{Spec}(\bz)$.  We have 
\begin{equation*}
\mathrm{det}(\xi_{\mathcal{X}/\mathbb{S},n})=\pm A(\mathcal{X})^{n-d/2} \cdot \frac{\zeta^*(\mathcal{X}_\infty,n)}{\zeta^*(\mathcal{X}_\infty,d-n)}.
\end{equation*}
\end{thm}

To conclude, we formulate the special value conjecture stated in \cite[Conjecture 1.1]{Flach-Morin20} in light of Corollary \ref{corCXnintro}. Let $\X$ be a regular scheme of pure dimension $d$ proper over $\mathrm{Spec}(\bz)$ and let $n\in \bz$. If $\X$ satisfies Assumptions $\mathbf{L}(\overline{\mathcal{X}}_{et},n)$, $\mathbf{L}(\overline{\mathcal{X}}_{et},d-n)$, $\mathbf{AV}(\overline{\mathcal{X}}_{et},n)$ and $\mathbf{B}(\mathcal{X},n)$ of \cite{Flach-Morin18}, then we may define a perfect complex of abelian groups $R\Gamma_{W,c}(\X,\bz(n))$, consider the fundamental line
\[ \Delta(\X/\bs,n):={\det}_\bz R\Gamma_{W,c}(\mathcal{X},\mathbb{Z}(n))\otimes_\bz {\det}_\bz R\Gamma(\X,L\Omega_{\X/\bs}^{<n})\]
and define a canonical trivialization
\begin{equation*}
\lambda:\br\xrightarrow{\sim}\Delta(\X/\bs,n)\otimes_\bz\br.\label{introtriv}
\end{equation*}
Let $\zeta(\X,s)$ be the Zeta function  of $\X$. Assuming that $\zeta(\X,s)$ has a meromorphic continuation to the entire complex plane, we denote by $\zeta^*(\X,n)\in\br^\times$ its leading Taylor coefficient at $s=n$.

\begin{conj} \label{conjmain} (Joint with Flach \cite{Flach-Morin20})
We have
$$
\lambda(\zeta^*(\mathcal{X},n)^{-1}\cdot\mathbb{Z})= \Delta(\mathcal{X}/\mathbb{S},n).
$$
\end{conj}
It follows from Theorem \ref{thmintro} that Conjecture \ref{conjmain} is compatible with the functional equation in the form of \cite[Conjecture 1.3]{Flach-Morin20}. If $\X$ lies over a finite field, then Conjecture \ref{conjmain}  follows from the (conjectured) finite generation of the Weil-étale motivic cohomology groups $H^{i}_W(\X,\bz(n))$, which in turn follows from the Tate conjecture. If $n=1$, $d=2$ and $\X/\bz$ is flat, then Conjecture \ref{conjmain} is equivalent to the Birch and Swinnerton-Dyer conjecture by \cite[Section 5]{Flach-Siebel19}.  Finally, if $\X/\mathcal{O}_F$ is smooth over a number ring and $n\leq 1$, then  Conjecture \ref{conjmain} is equivalent to the Bloch-Kato conjecture \cite{Bloch-Kato88}  for the motive $h(\X_{\bq})(n)$ in the formulation of Fontaine and Perrin-Riou \cite{Fontaine-Perrin-Riou-91}. Therefore, if $\X/\mathcal{O}_F$ is smooth over a number ring and $n\geq d-1$, then  Conjecture \ref{conjmain} follows from the functional equation together with the Bloch-Kato conjecture. Finally, the relationship between Conjecture \ref{conjmain} and Deninger's conjectures (see e.g. \cite{Deninger94}) is explained in \cite[Section 5.2]{Flach-Morin20a}.\\

\hspace{-0.5cm}{\bf{Aknowelements.}} I am grateful to  Spencer Bloch, Christopher Deninger, Lars Hesselholt, Steve Lichtenbaum, Matthew Morrow and Niranjan Ramachandran for interesting comments about this paper. Special thanks are due to Matthias Flach for many discussions related to the correction factor.

\section{Notations}

We use the theory of $\infty$-categories developed in \cite{LurieHTT} and \cite{LurieHA}, and try to give precise references to these books whenever it feels useful. 

\subsection{Filtrations}

We refer to \cite{Gwilliam-Pavlov18} and \cite[Section 5.1]{BMS} for the following statements. For any $\mathbb{E}_{\infty}$-ring $R$ we denote by $D(R)$ the derived $\infty$-category of $R$ and by $DF(R):=\mathrm{Fun}(\bz^{\mathrm{op}},D(R))$ the filtered derived $\infty$-category of $R$, which is a symmetric monoidal presentable stable $\infty$-category via the Day convolution. Note that, if $R=\bs$ is the sphere spectrum, then $D(\bs)=\mathrm{Sp}$ is the $\infty$-category of spectra. A (decreasing $\bz^{\mathrm{op}}$-indexed) filtration $X^*$ in $D(R)$ is an object $X^*\in DF(R)$. A filtration $X^*$ in $D(R)$ is said to be $\bn^{\mathrm{op}}$-indexed if
$$\mathrm{gr}^n(X^*):=\mathrm{Cofib}(X^{n+1}\rightarrow X^n)\simeq 0$$
for any $n<0$.  A filtration $X^*$ is  said to be complete if $\mathrm{lim}\, X^*\simeq 0$. We denote by $\widehat{DF}(R)$ the full subcategory of $D(R)$ spanned by the complete filtrations. A filtration $X^*$ on $X\in D(R)$ is a filtration $X^*$ together with a map $\mathrm{colim}\, X^*\rightarrow X$. A filtration $X^*$ on $X\in D(R)$ is said to be exhaustive if the map $\mathrm{colim}\, X^*\rightarrow X$ is an equivalence. A filtration $X^*\in DF(R)$ is said to be multiplicative if it is equipped with the structure of an $\mathbb{E}_{\infty}$-algebra object in $DF(R)$. The inclusion 
$\widehat{DF}(R)\rightarrow DF(R)$
has a left adjoint, called the completion functor, which we denote by $X^*\mapsto \widehat{X^*}$. Moreover $\widehat{DF}(R)$ has  a symmetric monoidal structure such that the completion functor is symmetric monoidal. Explicitly, we have
$$X^*\widehat{\otimes} Y^*=\widehat{X^*\otimes Y^*}.$$
In the above definitions, we may replace $D(R)$ by any symmetric monoidal presentable stable $\infty$-category. A bifiltration $X^{**}$ in $D(R)$ is a filtration in $DF(R)$. A bifiltration $X^{**}$ is said to be bicomplete if it is a complete filtration in $\widehat{DF}(R)$. A (resp. biexhaustive) bifiltration $X^{**}$ on $X\in D(R)$ is a (resp. exhaustive) filtration on a (resp. exhaustive) filtration on $X$. The $\infty$-category of bifiltrations $DBF(R)$ in  $D(R)$ is equivalent to the presentable stable $\infty$-category $\mathrm{Fun}(\bz^{\mathrm{op}}\times \bz^{\mathrm{op}},D(R))$, which is also symmetric monoidal via Day convolution. A bifiltration $X^{**}$ is said to be multiplicative if it is equipped with the structure of an $\mathbb{E}_{\infty}$-algebra object in $DBF(R)$, i.e. if $X^{**}$ is a multiplicative filtration in $DF(R)$. We also denote by $\widehat{DBF}(R)$ the $\infty$-category of bicomplete bifiltrations.

\subsection{$p$-adic completion}

We refer to \cite[Section 7.3]{LurieSAG} for the following statements. If $R$ is an $\mathbb{E}_{\infty}$-ring and $p$ is a prime number, we may consider the full subcategory $D(R)^{\mathrm{Cpl}(p)}$ of $D(R)$ consisting of $p$-complete $R$-modules. Here $(p)$ is the ideal generated by $p$ in the ring $\pi_0(R)$.
The inclusion $D(R)^{\mathrm{Cpl}(p)}\rightarrow D(R)$ has a left adjoint, called the $p$-completion functor, which we denote by
$$\appl{D(R)}{D(R)^{\mathrm{Cpl}(p)}}{X}{X^{\wedge}_p}$$
Note that if $R$ is a (discrete) commutative ring and $X\in D(R)$ then $$X^{\wedge}_p\simeq R\mathrm{lim}_{\nu}\,(X\otimes_{\bz}^L{\bz}/p^{\nu}\bz).$$ 
Then $D(R)^{\mathrm{Cpl}(p)}$ has a symmetric monoidal structure $\widehat{\otimes}$ such that the $p$-completion functor is symmetric monoidal \cite[Variant 7.3.5.6]{LurieSAG}.   Explicitly, for $X, Y\in D(R)^{\mathrm{Cpl}(p)}$, we have
$$X\widehat{\otimes} Y\simeq (X\otimes Y)^{\wedge}_p$$
and the unit object for $\widehat{\otimes}$ is $R^{\wedge}_p$.  The inclusion $D(R)^{\mathrm{Cpl}(p)}\rightarrow D(R)$, being right adjoint to a symmetric monoidal functor, is lax symmetric monoidal. In particular, if $X$ is an $\mathbb{E}_{\infty}$-$R$-algebra, then so is $X^{\wedge}_p$, where we also denote by $X^{\wedge}_p$ its image in $D(R)$. Finally we denote by $X^{\wedge}=\prod_p X^{\wedge}_p$ the profinite completion of $X\in D(R)$, where the product is taken over all prime numbers.

\subsection{Topological Hochschild homology}

If  $\mathcal{D}^{\otimes}$ is a  symmetric monoidal $\infty$-category, then we denote by $\mathrm{CAlg}(\mathcal{D})$  the $\infty$-category of $\mathbb{E}_{\infty}$-algebras in $\mathcal{D}$. Let $R$ be an $\mathbb{E}_{\infty}$-ring spectrum, and let $A$ be an $\mathbb{E}_{\infty}$-$R$-algebra, i.e. $A\in \mathrm{CAlg}(D(R))$. One considers the constant functor 
$$T_{A/R}:\bt\rightarrow \mathrm{CAlg}(D(R))$$
with value $A$, and one defines
$$THH(A/R):=A^{\otimes_R\bt}:=\mathrm{colim} \,T_{A/R}.$$
Then $THH(A/R)$ is a $\bt$-equivariant $\mathbb{E}_{\infty}$-$R$-algebra, there is a canonical (non-equivariant) map $A\rightarrow THH(A/R)$ of $\mathbb{E}_{\infty}$-$R$-algebra, and $THH(A/R)$ is universal for these properties. 
We have
\begin{equation}\label{computeTHH1}
THH(A/R)\otimes^L_RS\simeq THH(A\otimes^L_RS/S)
\end{equation}
where $\otimes^L_R$ denotes the coproduct in $\mathrm{CAlg}(D(R))$, which we sometimes simply denote by $\otimes_R$. The equivalence (\ref{computeTHH1}) follows from the fact that the left hand side satisfies the universal property of the right hand side.

If $R=\bs$ is the sphere spectrum, we set $THH(A):=THH(A/\bs)$. If both $A$ and $R$ are discrete, so that they can be identified with classical commutative rings, then we set
$$HH(A/R):=THH(A/R)$$
and we have
\begin{equation}\label{computeTHH2}
THH(A)\otimes_{THH(R)}R\simeq HH(A/R).
\end{equation}
If $A$ is discrete and $R=\bz$, then we set $HH(A):=HH(A/\bz)$.

\subsection{Quasisyntomic rings} We recall some definitions introduced in \cite[Section 4]{BMS}. Let $A$ be a commutative ring and let $K\in D(A)$. Recall that $K$ is said to have Tor-amplitude in $[a,b]$ if $K\otimes^L_AM$ is cohomologically concentrated in degrees in $[a,b]$ for any $A$-module $M$, and that $K$ is said to have $p$-complete Tor-amplitude in $[a,b]$ if $K\otimes^L_AA/pA\in D(A/pA)$ has Tor-amplitude in $[a,b]$. We say that $K\in D(A)$ is $p$-completely flat if $K$ has $p$-complete Tor-amplitude in $[0,0]$. Recall also that a ring $A$ is said to have bounded $p^{\infty}$-torsion if the inductive system $A[p^{\nu}]:=\mathrm{Ker}(A\stackrel{p^{\nu}}{\rightarrow} A)$, indexed by $\nu\geq 1$, is eventually constant. The ring $A$ is said to be $p$-completely quasi-syntomic if $A$ is $p$-complete, has bounded $p^{\infty}$-torsion, and $L_{A/\bz_p}\in D(A)$ has $p$-complete amplitude in $[-1,0]$.  A commutative ring $S$ is quasiregular semiperfectoid if $S$ is $p$-completely quasi-syntomic and if there exists a surjective morphism $S'\rightarrow S$ where $S'$ is perfectoid.

We denote by $\mathrm{QSyn}_{\bz_p}$ the category of $p$-completely quasisyntomic rings, and by $\mathrm{QRSPerfd}_{\bz_p}$ the category of $p$-complete quasiregular semiperfectoid rings $S$. If $R\in \mathrm{QSyn}_{\bz_p}$, then $\mathrm{QSyn}_{\bz_p}:=(\mathrm{QSyn}_{\bz_p})_{R/}$ denotes the category of $p$-completely quasisyntomic rings $A$ endowed with a map $R\rightarrow A$, and we define similarly $\mathrm{QRSPerfd}_{R}$.  Finally we denote by $\mathrm{qSyn}_{R}$ (respectively $\mathrm{qrsPerfd}_{R}$) the full subcategory of $\mathrm{QSyn}_{R}$ (respectively $\mathrm{QRSPerfd}_{R}$) spanned by the rings $A$ such that the structure map $R\rightarrow A$ is quasisyntomic \cite[Definition 4.10]{BMS}.

We say that a commutative ring $A$ is quasi-lci if $L_{A/\bz}$ has Tor-amplitude in $[-1,0]$. We say that $A$ has bounded torsion if $A$ has bounded $p^{\infty}$-torsion for any prime $p$.

\begin{lem}
Let $A$ be a quasi-lci ring with bounded torsion. Then for any prime $p$, $A^{\wedge}_p$ is discrete and $p$-completely quasi-syntomic.
\end{lem}
\begin{proof} Let $p$ be a prime number. Since $A$ has bounded $p^{\infty}$-torsion, the projective systems $\{A\otimes^L_{\bz} \bz/p^{\nu}\bz\}$ and $\{A/p^{\nu}\}$ are pro-isomorphic. Moreover $\{A/p^{\nu}\}$ is Mittag-Leffler, hence  $$A^{\wedge}_p= R\mathrm{lim} \, A\otimes^L_{\bz} \bz/p^{\nu}\bz\simeq \mathrm{lim}\,A/p^{\nu}$$
is discrete and $p$-complete. For any $\nu\geq1$, the canonical map
\begin{equation}\label{forcotangentmodulo}
A\otimes^L_{\bz} \bz/p^{\nu}\bz\stackrel{\sim}{\rightarrow} A^{\wedge}_p\otimes^L_{\bz} \bz/p^{\nu}\bz
\end{equation}
is an equivalence since $A^{\wedge}_p$  is the derived $p$-completion of $A$. This gives

$$A^{\wedge}_p[p^{\nu}]\simeq A[p^{\nu}]
\textrm{ and }
A/pA\simeq A^{\wedge}_p/pA^{\wedge}_p.$$
In particular, $A^{\wedge}_p$ has bounded $p$-torsion. Moreover, (\ref{forcotangentmodulo}) gives
$$L_{A/\bz}\otimes^L _{\bz}\bz/p\bz\simeq L_{A^{\wedge}_p/\bz_p}\otimes^L _{\bz}\bz/p\bz.$$
After derived base change along $A\otimes^L_{\bz}\bz/p\bz\rightarrow A/p$, we obtain
$$L_{A/\bz}\otimes^L_A A/pA\simeq L_{A^{\wedge}_p/\bz_p}\otimes^L_A A/pA .$$
Similarly, (\ref{forcotangentmodulo}) and derived base change give 
$$A/pA \simeq A \otimes^L_A A/pA \stackrel{\sim}{\rightarrow} A^{\wedge}_p \otimes^L_A A/pA$$
hence we have
$$L_{A/\bz}\otimes^L_A A/pA\simeq L_{A^{\wedge}_p/\bz_p}\otimes^L_{A^{\wedge}_p} A^{\wedge}_p \otimes^L_A A/pA\simeq L_{A^{\wedge}_p/\bz_p}\otimes^L_{A^{\wedge}_p} A^{\wedge}_p/pA^{\wedge}_p.$$
But $L_{A/\bz}\otimes^L_{A} A/pA \in D(A/pA)$ has Tor-amplitude in $[-1,0]$ by stability of Tor-amplitude by derived base change. Hence $L_{A^{\wedge}_p/\bz_p}\in D(A^{\wedge}_p)$ has $p$-complete amplitude in $[-1,0]$.

\end{proof}

\subsection{The BMS filtrations}\label{sect-BMS}

We consider the sites given by the categories $\mathrm{QSyn}^{\mathrm{op}}_{\bz_p}$  and $\mathrm{QRSPerfd}_{\bz_p}^{\mathrm{op}}$ endowed with the quasisyntomic topology, see \cite[Lemma 4.17 and Lemma 4.27]{BMS}. 
Let $\mathcal{C}$ be a presentable category. Then the restriction along $\mathrm{QRSPerfd}_{\bz_p}^{\mathrm{op}}\rightarrow \mathrm{QSyn}^{\mathrm{op}}_{\bz_p}$ induces an equivalence \cite[Proposition 4.31]{BMS} 
$$\mathrm{Sh}_{\mathcal{C}}(\mathrm{QSyn}^{\mathrm{op}}_{\bz_p})\longrightarrow \mathrm{Sh}_{\mathcal{C}}(\mathrm{QRSPerfd}^{\mathrm{op}}_{\bz_p})$$
between the corresponding $\infty$-categories of $\mathcal{C}$-valued sheaves. We denote an inverse to this functor by 
\begin{equation}\label{unfolding}
\appl{\mathrm{Sh}_{\mathcal{C}}(\mathrm{QRSPerfd}^{\mathrm{op}}_{\bz_p})}{\mathrm{Sh}_{\mathcal{C}}(\mathrm{QSyn}^{\mathrm{op}}_{\bz_p})}{\mathcal{F}}{R\Gamma_{\mathrm{syn}}(-,\mathcal{F})},
\end{equation}
unlike in \cite{BMS}, where it is denoted by $\mathcal{F}\rightarrow \mathcal{F}^{\beth}$. Consider the $\widehat{DF}(\bz_p)$-valued presheaves $\tau_{\geq 2*} HH(-/\bz_p,\bz_p)$, $\tau_{\geq 2*} HP(-/\bz_p,\bz_p)$ and $\tau_{\geq 2*} HC^-(-/\bz_p,\bz_p)$ on $\mathrm{QRSPerfd}_{\bz_p}^{\mathrm{op}}$.  For $?=H,P,C^-$,  the presheaf $\tau_{\geq 2*} H?(-/\bz_p,\bz_p)$ is a  sheaf for the quasisyntomic topology \cite[Section 5.2]{BMS}, and
one defines 
$$\mathrm{Fil}^*_{BMS}H?(-/\bz_p,\bz_p):=R\Gamma_{\mathrm{syn}}(-,\tau_{\geq 2*} H?(-/\bz_p,\bz_p))$$
which is a  $\widehat{DF}(\bz_p)$-valued sheaf on $\mathrm{QSyn}_{\bz_p}^{\mathrm{op}}$. 
Note that
$$\mathrm{Fil}^*_{BMS}H?(S/\bz_p,\bz_p)\simeq \tau_{\geq 2*} H?(S/\bz_p,\bz_p) $$
for any $S\in \mathrm{QRSPerfd}_{\bz_p}$. Since our base ring is $R=\bz_p$, the proof of \cite[Theorem 1.17]{BMS} is valid for any $A\in \mathrm{QSyn}_{\bz_p}$. In particular, for any $A\in \mathrm{QSyn}_{\bz_p}$, the $n$-graded piece 
$\mathrm{gr}^n_{BMS}H?(A/\bz_p,\bz_p)$ is identified with $(L\Lambda^nL_{A/\bz_p})^{\wedge}_p[n]$, $(L\widehat{\Omega}_{A/\bz_p})^{\wedge}_p[2n]$ and  $(L\widehat{\Omega}^{\geq n}_{A/\bz_p})^{\wedge}_p[2n]$ for $?=H,P,C^-$ respectively. 

In the previous paragraph, one may replace the base ring $\bz_p$ by an arbitrary $R\in \mathrm{QSyn}_{\bz_p}$, provided we also replace  $\mathrm{QSyn}_{R}$ and $\mathrm{QRSPerfd}_{R}$ by $\mathrm{qSyn}_{R}$ and $\mathrm{qrsPerfd}_{R}$ respectively.

Similarly, consider the  $\widehat{DF}(\bs)$-valued presheaves $\tau_{\geq 2*} THH(-,\bz_p)$, $\tau_{\geq 2*} TP(-,\bz_p)$ and $\tau_{\geq 2*} TC^-(-,\bz_p)$ on $\mathrm{QRSPerfd}_{\bz_p}^{\mathrm{op}}$. For $?=HH,P,C^-$, the presheaf $\tau_{\geq 2*} T?(-/\bz_p,\bz_p)$ is a  sheaf for the quasisyntomic topology \cite[Section 7]{BMS}, and
one defines 
$$\mathrm{Fil}^*_{BMS}T?(-,\bz_p):=R\Gamma_{\mathrm{syn}}(-,\tau_{\geq 2*} T?(-,\bz_p)),$$
which is a $\widehat{DF}(\bs)$-valued sheaf on $\mathrm{QSyn}_{\bz_p}^{\mathrm{op}}$. One has
$$\mathrm{Fil}^*_{BMS}T?(S,\bz_p)\simeq \tau_{\geq 2*} T?(S,\bz_p) $$
for any $S\in \mathrm{QRSPerfd}_{\bz_p}$.

\section{The motivic filtration}

\subsection{Cartesian squares}

We define a filtration on $THH$ as follows. For any commutative ring $A$ we define a $\bt$-equivariant filtration
$$\mathcal{Z}^*THH(A):= THH(A)\otimes_{THH(\bz)} \tau_{\geq 2*-1}THH(\bz)$$
on $THH(A)$. Recall from \cite{Bökstedt85} and \cite{Lindenstrauss-Madsen00} that $\pi_{i}THH(\bz)$ is $\bz$ for $i=0$ and $\bz/n\bz$ for $i=2n-1>0$ and $0$ else. This yields $\bt$-equivariant equivalences of spectra
$$\mathrm{gr}_{\mathcal{Z}}^0 (THH(A))\simeq THH(A)\otimes_{THH(\bz)} \bz\simeq HH(A),$$
and
\begin{eqnarray*}
\mathrm{gr}_{\mathcal{Z}}^n (THH(A))&\simeq& THH(A)\otimes_{THH(\bz)} \bz/n[2n-1]\\
&\simeq& THH(A)\otimes_{THH(\bz)}\bz\otimes_{\bz} \bz/n[2n-1] \\
&\simeq & HH(A)\otimes^L_{\bz} \bz/n[2n-1]
\end{eqnarray*}
for $n>0$, and $\mathrm{gr}_{\mathcal{Z}}^n (THH(A))\simeq 0$ for $n<0$. We define similarly 
$$\mathcal{Z}^nTHH(A,\bz_p):= THH(A,\bz_p)\otimes_{THH(\bz)} \tau_{\geq 2n-1}THH(\bz)$$
Using \cite[Lemma 2.5]{BMS}, we also have 
$$\mathrm{gr}_{\mathcal{Z}}^0 (THH(A,\bz_p))\simeq HH(A,\bz_p)$$
and
$$\mathrm{gr}_{\mathcal{Z}}^n (THH(A,\bz_p))\simeq HH(A,\bz_p)\otimes^L_{\bz} \bz/n[2n-1]$$
for $n>0$.

\begin{prop}\label{prop-filt-Z}
Let $A$ be  a commutative ring. The filtrations ${\mathcal{Z}}^*THH(A)$ and ${\mathcal{Z}}^*THH(A,\bz_p)$, on $THH(A)$ and $THH(A,\bz_p)$ respectively, are both complete and exhaustive.
\end{prop}
\begin{proof}
The filtration ${\mathcal{Z}}^*THH(A)$ is $\bn^{\mathrm{op}}$-indexed and we have $$\mathrm{colim}\,{\mathcal{Z}}^*THH(A)\simeq {\mathcal{Z}}^0THH(A)\simeq THH(A).$$
so that ${\mathcal{Z}}^*THH(A)$ is exhaustive. The functor
$$
THH(A)\otimes_{THH(\bz)}(-):D(THH(\bz))\rightarrow \mathrm{Sp}
$$
is right $t$-exact \cite[Definition 1.3.3.1, Corollary 7.2.1.23]{LurieHA} hence we have ${\mathcal{Z}}^nTHH(A)\in \mathrm{Sp}_{\geq 2n-1}$. It follows that the filtration ${\mathcal{Z}}^*THH(A)$ is complete, by left completeness of $\mathrm{Sp}$. The same argument applies to ${\mathcal{Z}}^*THH(A,\bz_p)$ since $THH(A,\bz_p)$ is connective, as the $p$-completion is right $t$-exact too \cite[Proposition 7.3.4.4]{LurieSAG}.

\end{proof}

We refer to \cite[Section 2.2]{BMS} for the Hochschild-Kostant-Rosenberg filtration, which we denote by $\mathrm{Fil}_{HKR}^*$.
\begin{lem}\label{lem-complete-HH}
Let $p$ be a prime number and let $A$ be a commutative ring with bounded $p^{\infty}$-torsion. The map
 $$\left(\mathrm{Fil}_{HKR}^*HH(A)\right)^{\wedge}_p\rightarrow \left(\mathrm{Fil}_{HKR}^*HH(A^{\wedge}_p/\bz_p)\right)^{\wedge}_p$$
is an equivalence of complete multiplicative $\bt$-equivariant filtrations.
\end{lem}
\begin{proof} 
Since $A$ has bounded $p^{\infty}$-torsion, its derived $p$-adic completion $A^{\wedge}_p:=R\mathrm{lim}(A\otimes^L_{\bz}{\bz}/p^{\bullet})$ coincides with its naive $p$-adic completion. In particular, $A^{\wedge}_p$ is a (discrete) commutative ring. The map of the lemma is induced by the morphism of filtrations 
 $$\mathrm{Fil}_{HKR}^*HH(A)\rightarrow \mathrm{Fil}_{HKR}^*HH(A^{\wedge}_p/\bz_p)$$
The $p$-adic completion functor $(-)^{\wedge}_p$ commutes with small limits, hence
$$\mathrm{lim}\,\left(\mathrm{Fil}_{HKR}^*HH(A)\right)^{\wedge}_p\simeq \left(\mathrm{lim}\,\mathrm{Fil}_{HKR}^*HH(A)\right)^{\wedge}_p\simeq 0$$
since the $HKR$-filtration is complete. Similarly we have  $$\mathrm{lim}\,\left(\mathrm{Fil}_{HKR}^*HH(A^{\wedge}_p/\bz_p)\right)^{\wedge}_p\simeq 0.$$ 
Hence the map of the Lemma is a map of complete filtrations, so that it is enough to show that the induced map
$$\left(\mathrm{gr}_{HKR}^{i}HH(A)\right)^{\wedge}_p\rightarrow \left(\mathrm{gr}_{HKR}^{i}HH(A^{\wedge}_p/\bz_p)\right)^{\wedge}_p$$
is an equivalence for all $i\geq 0$. But this map may be identified with the natural map
$$\left(L\Lambda^{i}L_{A/\bz}[i]\right)^{\wedge}_p\rightarrow \left(L\Lambda^{i}L_{A^{\wedge}_p/\bz_p}[i]\right)^{\wedge}_p$$
which is an equivalence since we have
\begin{eqnarray*}
\left(L\Lambda^{i}L_{A/\bz}\right)\otimes^L_{\bz}{\bz}/p^{\bullet}&\simeq& L\Lambda^{i}L_{(A\otimes_\bz^L\bz/p^{\bullet})/(\bz/p^{\bullet})}\\
& \simeq& L\Lambda^{i}L_{(A^{\wedge}_p\otimes_{\bz}^L\bz/p^{\bullet})/(\bz/p^{\bullet})}\\
&\simeq &\left(L\Lambda^{i}L_{A^{\wedge}_p/\bz_p}\right)\otimes^L_{\bz}{\bz}/p^{\bullet}.
\end{eqnarray*}
as the map $A\otimes_\bz^L\bz/p^{\nu}\rightarrow A^{\wedge}_p\otimes_\bz^L\bz/p^{\nu}$ is an equivalence in the $\infty$-category of simplicial commutative rings.
\end{proof}

\begin{lem}\label{lem-complete-THH}
Let $p$ be a prime number and let $A$ be a commutative ring with bounded $p^{\infty}$-torsion. Then the map
 \begin{equation}\label{maphere}
 THH(A,\bz_p)\rightarrow THH(A^{\wedge}_p,\bz_p)
 \end{equation}
is an equivalence of $\bt$-equivariant $\mathbb{E}_{\infty}$-ring spectra.
\end{lem}
\begin{proof}
It is enough to check that (\ref{maphere}) is an equivalence of spectra. The map (\ref{maphere}) induces a map of complete filtrations from
$$\mathcal{Z}^*THH(A,\bz_p):=THH(A,\bz_p)\otimes_{THH(\bz)} \tau_{\geq 2*-1}THH(\bz)$$
to
$$\mathcal{Z}^*THH(A^{\wedge}_p,\bz_p):=THH(A^{\wedge}_p,\bz_p)\otimes_{THH(\bz)} \tau_{\geq 2*-1}THH(\bz).$$
The induced map
$$\mathrm{gr}^0_{\mathcal{Z}}THH(A,\bz_p)\simeq HH(A,\bz_p) \rightarrow HH(A^{\wedge}_p,\bz_p)\simeq \mathrm{gr}^0_{\mathcal{Z}}THH(A^{\wedge}_p,\bz_p)$$
is an equivalence by Lemma \ref{lem-complete-HH}. For $n>0$, the induced map
$$\mathrm{gr}^n_{\mathcal{Z}}THH(A,\bz_p)\rightarrow \mathrm{gr}^n_{\mathcal{Z}}THH(A^{\wedge}_p,\bz_p)$$
identifies with the map from
$$THH(A,\bz_p)\otimes_{THH(\bz)}\pi_{2n-1}THH(\bz)[2n-1]\simeq HH(A,\bz_p)\otimes_{\bz}\bz/n[2n-1]$$
to
$$THH(A^{\wedge}_p,\bz_p)\otimes_{THH(\bz)}\pi_{2n-1}THH(\bz)[2n-1]\simeq HH(A^{\wedge}_p,\bz_p)\otimes_{\bz}\bz/n[2n-1]$$
which is also an equivalence by Lemma \ref{lem-complete-HH}.
\end{proof}

\begin{rem}\label{rem-p-general}
Lemma \ref{lem-complete-THH} holds for any $\mathbb{E}_{\infty}$-ring spectrum $A$.  Indeed, consider the composite functor
$$\bt\stackrel{T_A}{\longrightarrow}\mathrm{CAlg}(D(\bs))\stackrel{(-)^{\wedge}_p}{\longrightarrow}\mathrm{CAlg}(D(\bs)^{\mathrm{Cpl}(p)})$$
where $T_A$ is the constant functor with value $A$, and $(-)^{\wedge}_p$ is the $p$-completion functor.
Then we have
$$THH(A,\bz_p):=(\mathrm{colim}\,T_A)^{\wedge}_p\simeq \mathrm{colim}\left((-)^{\wedge}_p\circ T_A\right)$$
$$\simeq \mathrm{colim}\left((-)^{\wedge}_p\circ T_{A^{\wedge}_p}\right)\simeq THH(A^{\wedge}_p,\bz_p).$$
 \end{rem}

\begin{prop}\label{decompoTHH}
Let $A$ be a commutative ring with bounded torsion. The diagram of $\bt$-equivariant $\mathbb{E}_{\infty}$-ring spectra 
\[ \xymatrix{
THH(A)\ar[d]\ar[r]& HH(A) \ar[d]\\
\prod_pTHH(A^{\wedge}_p,\bz_p)\ar[r]&\prod_pHH(A^{\wedge}_p,\bz_p) 
}
\]
is cartesian.
\end{prop}
\begin{proof}
By \cite[Corollary 3.2.2.4]{LurieHA} it is enough to show that the square of the proposition is a pull-back square of spectra. First we show that 
\[ \xymatrix{
THH(A)\ar[d]\ar[r]& HH(A) \ar[d]\\
THH(A)^{\wedge}\ar[r]&HH(A)^{\wedge} 
}
\]
is cartesian. It is enough to show that the induced map from
$$\mathrm{Fib}\left(THH(A)\rightarrow HH(A)\right)$$
to
$$\mathrm{Fib}\left(THH(A)^{\wedge}\rightarrow HH(A)^{\wedge}\right)\simeq \left(\mathrm{Fib}
(THH(A)\rightarrow HH(A))\right)^{\wedge}$$
is an equivalence. Noting that $\mathrm{Fib}\left(THH(A)\rightarrow HH(A)\right)\simeq \mathcal{Z}^1THH(A)$ we need to show that
$$\mathcal{Z}^1THH(A)\rightarrow \left(\mathcal{Z}^1THH(A)\right)^{\wedge}$$
is an equivalence. There is a  complete exhaustive filtration $\mathcal{Z}^*$ on $\mathcal{Z}^1THH(A)$ such that  
$$\mathrm{gr}_{\mathcal{Z}}^n (\mathcal{Z}^1THH(A))\simeq HH(A)\otimes^L_{\bz} \bz/n[2n-1]$$
for all $n>0$ and $\mathrm{gr}_{\mathcal{Z}}^n (\mathcal{Z}^1T(A))=0$ for $n\leq0$. Hence it is enough to check that the map
$$HH(A)\otimes^L_{\bz} \bz/n[2n-1]\rightarrow \left(HH(A)\otimes^L_{\bz} \bz/n[2n-1]\right)^{\wedge}$$
is an equivalence for any $n>0$, which is clear. The result follows since the maps
$$HH(A)^{\wedge}\rightarrow \prod_pHH(A^{\wedge}_p,\bz_p)
\textrm{ and } 
THH(A)^{^{\wedge}}\rightarrow \prod_pTHH(A^{\wedge}_p,\bz_p)$$
are equivalences by Lemma \ref{lem-complete-HH} and Lemma \ref{lem-complete-THH} respectively.

\end{proof}

Either one of the following lemmas suffices to prove Proposition \ref{propfibproduct} below.
\begin{lem}\label{sauvelemma}
The canonical map of spectra
$$(\mathcal{Z}^1THH(A))_{h\bt}\longrightarrow \Prod_p \left((\mathcal{Z}^1THH(A))^{\wedge}_p\right)_{h\bt}\simeq \left((\mathcal{Z}^1THH(A))_{h\bt}\right)^{\wedge}$$
is an equivalence.
\end{lem}
\begin{proof}
The complete exhaustive $\bn^{\mathrm{op}}_{>0}$-indexed filtration $\mathcal{Z}^*$ on $\mathcal{Z}^1THH(A)$ of the previous proof induces a complete exhaustive filtration $(\mathcal{Z}^*)_{h\bt}$ (respectively $\prod_p((\mathcal{Z}^*)^{\wedge}_p)_{h\bt}$) on the source (respectively the target) of the map of the Lemma. Indeed, the completeness of these filtrations follow from the fact that both $(-)_{h\bt}$ and $(-)^{\wedge}_{p}$ are right $t$-exact, as the connectivity of $\mathcal{Z}^*$ tends to infinity. The map of complete filtrations $(\mathcal{Z}^*)_{h\bt}\rightarrow \prod_p((\mathcal{Z}^*)^{\wedge}_p)_{h\bt}$ induces 
$$(\mathrm{gr}^n_{\mathcal{Z}})_{h\bt}\rightarrow \prod_p((\mathrm{gr}^n_{\mathcal{Z}})^{\wedge}_p)_{h\bt}$$
on $n$-th graded pieces, which identifies with the equivalence
$$HC(A)\otimes^L_{\bz} \bz/n[2n-1]\rightarrow \prod_p(HC(A)\otimes^L_{\bz} \bz/n[2n-1])^{\wedge}_p$$
for any $n>0$.
\end{proof}

\begin{lem}\label{sauvelemma2}
The map $$\left(\prod_pTHH(A^{\wedge}_p,\bz_p)\right)_{h\bt}\rightarrow \prod_pTHH(A^{\wedge}_p,\bz_p)_{h\bt}\simeq \Prod_pTC^+(A^{\wedge}_p,\bz_p) $$
is an equivalence, and similarly for $HH$.
\end{lem}
\begin{proof}

Applying \cite[Lemma 3.3]{BMS} to the weak Postnikov tower $\{\prod_p \tau_{\leq n}THH(A^{\wedge}_p,\bz_p)\}$, we obtain 
\begin{eqnarray*}
\left(\prod_pTHH(A^{\wedge}_p,\bz_p)\right)_{h\bt}&\simeq& \left(\prod_p\mathrm{lim}_{n} \tau_{\leq n}THH(A^{\wedge}_p,\bz_p)\right)_{h\bt}\\
&\simeq& \left(\mathrm{lim}_{n} \prod_p \tau_{\leq n}THH(A^{\wedge}_p,\bz_p)\right)_{h\bt}\\
&\simeq& \mathrm{lim}_{n} \left( \prod_p \tau_{\leq n}THH(A^{\wedge}_p,\bz_p)\right)_{h\bt}.
\end{eqnarray*}
Similarly, we have 
\begin{eqnarray*}
\prod_pTHH(A^{\wedge}_p,\bz_p)_{h\bt}\simeq  \prod_p\left(\mathrm{lim}_{n} \tau_{\leq n}THH(A^{\wedge}_p,\bz_p)\right)_{h\bt}\simeq  \mathrm{lim}_{n} \prod_p \left(\tau_{\leq n}THH(A^{\wedge}_p,\bz_p)\right)_{h\bt}
\end{eqnarray*}
hence it suffices to prove the result for $\tau_{\leq n}THH(A^{\wedge}_p,\bz_p)$. By induction and shifting, the result then follows from the fact that  the map
$$\left(\prod_p\pi_nTHH(A^{\wedge}_p,\bz_p)\right)_{h\bt}\rightarrow \prod_p\left(\pi_nTHH(A^{\wedge}_p,\bz_p)\right)_{h\bt}$$ 
is an equivalence for any fixed $n\geq 0$. The same argument works for $HH(-,\bz_p)$.

\end{proof}

\begin{prop}\label{propfibproduct}
Let $A$ be a commutative ring with bounded torsion. The squares of $\mathbb{E}_{\infty}$-ring spectra
\[ \xymatrix{
TC^-(A)\ar[d]\ar[r]& HC^-(A) \ar[d]\\
\prod_pTC^-(A^{\wedge}_p,\bz_p)\ar[r]&\prod_pHC^-(A^{\wedge}_p,\bz_p) 
}
\]
and
\[ \xymatrix{
TP(A)\ar[d]\ar[r]& HP(A) \ar[d]\\
\prod_pTP(A^{\wedge}_p,\bz_p)\ar[r]&\prod_pHP(A^{\wedge}_p,\bz_p) 
}
\]
are both cartesian. Similarly, the square of spectra 
\[ \xymatrix{
TC^+(A)\ar[d]\ar[r]& HC(A) \ar[d]\\
\prod_pTC^+(A^{\wedge}_p,\bz_p)\ar[r]&\prod_pHC(A^{\wedge}_p,\bz_p) 
}
\]
is cartesian.  
\end{prop}
\begin{proof}

The first and the second square are both commutative squares of $\mathbb{E}_{\infty}$-ring spectra because the functors $(-)^{h\bt}$ and $(-)^{t\bt}$ are both lax symmetric monoidal \cite[Corollary I.4.3]{Nikolaus-Scholze18}. It is enough to check that they are cartesian as squares of spectra by \cite[Corollary 3.2.2.4]{LurieHA}. Concerning the first square, this follows from Proposition \ref{decompoTHH} since the functor $(-)^{h\bt}$ preserves small limits, in particular with countable products and fiber products. It follows from Proposition \ref{decompoTHH}, the exactness of $(-)_{h\bt}$, and either Lemma \ref{sauvelemma} or Lemma \ref{sauvelemma2}, that the third square is a cartesian of spectra. Taking the cofiber of the norm map $\Sigma(-)_{h\bt}\rightarrow (-)^{h\bt}$, we see that the second square is also cartesian.

\end{proof}

\subsection{Morphisms of filtrations}

\begin{prop}\label{mapfiltHKR}
Let $p$ be a prime number and let $A$ be a commutative ring with bounded $p^{\infty}$-torsion such that $A^{\wedge}_p\in \mathrm{QSyn}_{\bz_p}$. Then there is a canonical morphism of complete multiplicative $\bt$-equivariant filtrations
$$\mathrm{Fil}_{HKR}^*HH(A)\rightarrow \mathrm{Fil}_{BMS}^*HH(A^{\wedge}_p/\bz_p,\bz_p)$$
which induces an equivalence after $p$-completion. 
\end{prop}
\begin{proof}
First we define an equivalence
$$ \left(\mathrm{Fil}_{HKR}^*HH(A^{\wedge}_p/\bz_p)\right)^{\wedge}_p\stackrel{\sim}{\longrightarrow} \mathrm{Fil}_{BMS}^*HH(A^{\wedge}_p/\bz_p,\bz_p).$$
The presheaves $\left(\mathrm{Fil}_{HKR}^*HH(-/\bz_p)\right)^{\wedge}_p$ and $\mathrm{Fil}_{BMS}^*HH(-/\bz_p,\bz_p)$ are sheaves
on $\mathrm{QSyn}^{\mathrm{op}}_{\bz_p}$ with values in $\widehat{DF}(\bz_p[\bt])$. So it is enough to define an equivalence
$$ \left(\mathrm{Fil}_{HKR}^*HH(R/\bz_p)\right)^{\wedge}_p\stackrel{\sim}{\longrightarrow} \mathrm{Fil}_{BMS}^*HH(R/\bz_p,\bz_p)\simeq \tau_{\geq 2*}HH(R/\bz_p,\bz_p)$$
functorial in $R\in \mathrm{QRSPerfd}_{\bz_p}$. But 
$$\left(\mathrm{gr}_{HKR}^{i}HH(R/\bz_p)\right)^{\wedge}_p\simeq \left(L\Lambda^{i}L_{R/\bz_p}\right)^{\wedge}_p[i]$$
is concentrated in homological degree $2i$ by \cite[Lemma 5.14(1)]{BMS}, hence $\left(\mathrm{Fil}_{HKR}^{i}HH(R/\bz_p)\right)^{\wedge}_p$ is concentrated in homological degrees $\geq 2i$ by induction. 
Hence the map $$\left(\mathrm{Fil}_{HKR}^{i}HH(R/\bz_p)\right)^{\wedge}_p\rightarrow HH(R/\bz_p)^{\wedge}_p=:HH(R/\bz_p,\bz_p)$$
factors through $\tau_{\geq 2i}HH(R/\bz_p,\bz_p)$. This gives a map of complete multiplicative $\bt$-equivariant filtrations  
$$ \left(\mathrm{Fil}_{HKR}^*HH(R/\bz_p)\right)^{\wedge}_p\longrightarrow \tau_{\geq 2*}HH(R/\bz_p,\bz_p)$$
which is an equivalence since
$$\left(\mathrm{gr}_{HKR}^{i}HH(R/\bz_p)\right)^{\wedge}_p\simeq \left(L\Lambda^{i}L_{R/\bz_p}\right)^{\wedge}_p[i] \stackrel{\sim}{\longrightarrow} (\pi_{2i}HH(R/\bz_p,\bz_p))[2i]$$
for any $i\geq 0$.

The morphism of the proposition is then given by the composite morphism
\begin{eqnarray*}
\mathrm{Fil}_{HKR}^*HH(A)&\longrightarrow& \left(\mathrm{Fil}_{HKR}^*HH(A)\right)^{\wedge}_p\\
&\stackrel{\sim}{\longrightarrow}& \left(\mathrm{Fil}_{HKR}^*HH(A^{\wedge}_p/\bz_p)\right)^{\wedge}_p\\
&\stackrel{\sim}{\longrightarrow}& \mathrm{Fil}_{BMS}^*HH(A^{\wedge}_p/\bz_p,\bz_p)
\end{eqnarray*}
where the first map is the $p$-adic completion map, the second map is given by  Lemma \ref{lem-complete-HH} and the third map is defined above. The result follows.

\end{proof}

In the statement below, we denote by $\mathrm{Fil}_{B}^*$ the filtration defined in \cite{Antieau18}.

\begin{prop}\label{prop-FilBpcompleteed}
Let $p$ be a prime number and let $A$ be a commutative ring with bounded $p^{\infty}$-torsion such that $A^{\wedge}_p\in \mathrm{QSyn}_{\bz_p}$. Then there are canonical maps of complete multiplicative filtrations
$$\mathrm{Fil}_{B}^*HC^{-}(A)\rightarrow \mathrm{Fil}_{BMS}^*HC^{-}(A^{\wedge}_p/\bz_p,\bz_p)$$
$$\mathrm{Fil}_{B}^*HP(A)\rightarrow \mathrm{Fil}_{BMS}^*HP(A^{\wedge}_p/\bz_p,\bz_p)$$
which induce equivalences after $p$-completion. 
\end{prop}
\begin{proof}
We treat the case of $HP$; the case of $HC^{-}$ is similar. We shall define a composite map of complete multiplicative filtrations
$$\mathrm{Fil}_{B}^*HP(A)\rightarrow \left(\mathrm{Fil}_{B}^*HP(A)\right)^{\wedge}_p\stackrel{\sim}{\rightarrow} \left(\mathrm{Fil}_{B}^*HP(A^{\wedge}_p/\bz_p)\right)^{\wedge}_p \stackrel{\sim}{\rightarrow} \mathrm{Fil}_{BMS}^*HP(A^{\wedge}_p/\bz_p,\bz_p).$$
where the first map is the $p$-completion map. First we consider the evident morphism of complete multiplicative filtrations
\begin{equation}\label{oneequi}
\left(\mathrm{Fil}_{B}^*HP(A)\right)^{\wedge}_p\longrightarrow \left(\mathrm{Fil}_{B}^*HP(A^{\wedge}_p/\bz_p)\right)^{\wedge}_p.
\end{equation}
In order to check that this morphism is an equivalence, it is enough to show that
$$\left(L\widehat{\Omega}_{A/\bz}[2i]\right)^{\wedge}_p\simeq \left(\mathrm{gr}_{B}^{i}HP(A)\right)^{\wedge}_p\longrightarrow \left(\mathrm{gr}_{B}^{i}HP(A^{\wedge}_p/\bz_p)\right)^{\wedge}_p\simeq \left(L\widehat{\Omega}_{A^{\wedge}_p/\bz_p}[2i]\right)^{\wedge}_p$$
is an equivalence, where we use \cite[Theorem 1.1(a)]{Antieau18}. Using the Hodge filtration, this follows from the fact that 
$$\left(L\Lambda^{i}L_{A/\bz}\right)^{\wedge}_p \rightarrow \left(L\Lambda^{i}L_{A^{\wedge}_p/\bz_p}\right)^{\wedge}_p$$
is an equivalence, see the proof of Lemma \ref{lem-complete-HH}. Hence (\ref{oneequi}) is an equivalence.

Now we define an equivalence
$$\left(\mathrm{Fil}_{B}^*HP(-/\bz_p)\right)^{\wedge}_p \stackrel{\sim}{\rightarrow} \mathrm{Fil}_{BMS}^*HP(-/\bz_p,\bz_p)$$
of $\widehat{DF}(\bz_p)$-valued sheaves on $\mathrm{QSyn}^{\mathrm{op}}_{\bz_p}$. First we notice that $\left(\mathrm{Fil}_{B}^*HP(-/\bz_p)\right)^{\wedge}_p$ is indeed a sheaf, since $\left(\mathrm{gr}_{B}^{i}HP(-/\bz_p)\right)^{\wedge}_p\simeq \left(L\widehat{\Omega}_{-/\bz}[2i]\right)^{\wedge}_p$ is a sheaf on $\mathrm{QSyn}^{\mathrm{op}}_{\bz_p}$, see \cite[Example 5.11]{BMS}. Moreover $\mathrm{Fil}_{BMS}^*HP(-/\bz_p,\bz_p)$ is a sheaf on $\mathrm{QSyn}^{\mathrm{op}}_{\bz_p}$ by definition. In view of  \cite[Proposition 4.31]{BMS}, it is therefore enough to define an equivalence
$$\left(\mathrm{Fil}_{B}^*HP(R/\bz_p)\right)^{\wedge}_p \stackrel{\sim}{\rightarrow} \tau_{\geq 2*} HP(R/\bz_p,\bz_p) \simeq \mathrm{Fil}_{BMS}^*HP(R/\bz_p,\bz_p)$$
functorial in $R\in \mathrm{QRSPerfd}_{\bz_p}$. We need to check that $\left(\mathrm{Fil}_{B}^{i}HP(R/\bz_p)\right)^{\wedge}_p$ is concentrated in homological degree $\geq 2i$. Since $R\in \mathrm{QRSPerfd}_{\bz_p}$, $\left(L\Lambda^{i}L_{R/\bz_p}\right)^{\wedge}_p[-i]$ is concentrated in degree $0$ by \cite[Lemma 5.14(1)]{BMS}. It follows by induction that $(L\Omega^{<n}_{R/\bz_p})^{\wedge}_p$ is concentrated in degree $0$ and that the transition morphisms are surjective on $H^0$, so that the derived limit $$(L\widehat{\Omega}_{R/\bz_p})^{\wedge}_p\simeq R\mathrm{lim}((L\Omega_{R/\bz_p}^{<n})^{\wedge}_p)$$ is concentrated in degree $0$ as well. Hence $\left(\mathrm{Fil}_{B}^{*}HP(R/\bz_p)\right)^{\wedge}_p$ is a complete filtration with $j$-graded piece 
$$\left(\mathrm{gr}_{B}^{j}HP(R/\bz_p)\right)^{\wedge}_p\simeq (L\widehat{\Omega}_{R/\bz_p})^{\wedge}_p[2j]$$
concentrated in homological degree $2j$. By induction and completeness of the filtration, it follows that $\left(\mathrm{Fil}_{B}^{i}HP(R/\bz_p)\right)^{\wedge}_p$ is indeed concentrated in homological degree $\geq 2i$. Hence the natural map $\left(\mathrm{Fil}_{B}^{i}HP(R/\bz_p)\right)^{\wedge}_p\rightarrow HP(R/\bz_p,\bz_p)$ factors through
$$\left(\mathrm{Fil}_{B}^{i}HP(R/\bz_p)\right)^{\wedge}_p\rightarrow \tau_{\geq 2i} HP(R/\bz_p,\bz_p)$$
and we get a morphism of complete multiplicative filtrations 
$$\left(\mathrm{Fil}_{B}^*HP(R/\bz_p)\right)^{\wedge}_p \rightarrow\tau_{\geq 2*} HP(R/\bz_p,\bz_p) \simeq \mathrm{Fil}_{BMS}^*HP(R/\bz_p,\bz_p).$$
This map is an equivalence since the induced map
$$\left(L\widehat{\Omega}_{R/\bz_p}[2i]\right)^{\wedge}_p\simeq \left(\mathrm{gr}_{B}^{i}HP(R/\bz_p)\right)^{\wedge}_p \rightarrow \mathrm{gr}_{BMS}^{i}HP(R/\bz_p,\bz_p)\simeq \left(L\widehat{\Omega}_{R/\bz_p}[2i]\right)^{\wedge}_p$$
is an equivalence for any $i\in \bz$ (see \cite[Theorem 1.17]{BMS} and Section \ref{sect-BMS}).

\end{proof}

\begin{prop}\label{mapBMS-T-H}
Let $A\in\mathrm{QSyn}_{\bz_p}$. There is a functorial morphism of multiplicative $\bt$-equivariant filtrations
$$\mathrm{Fil}_{BMS}^*THH(A,\bz_p)\rightarrow \mathrm{Fil}_{BMS}^*HH(A/\bz_p,\bz_p)$$
inducing the canonical map $THH(A,\bz_p) \rightarrow HH(A/\bz_p,\bz_p)$, and similarly for $TP$ and $TC^-$.
\end{prop}
\begin{proof}
The map $THH(-,\bz_p)\rightarrow HH(-/\bz_p,\bz_p)$ induces a morphism
$$\tau_{\geq 2*}THH(-,\bz_p)\rightarrow \tau_{\geq 2*}HH(-/\bz_p,\bz_p)$$
of $\widehat{DF}(\bs[\bt])$-valued sheaves on $\mathrm{QRSPerfd}_{\bz_p}^{\mathrm{op}}$ (more precisely, of sheaves of $\mathbb{E}_{\infty}$-algebras in $\widehat{DF}(\bs[\bt])$). We obtain
$$R\Gamma_{\mathrm{syn}}(A,\tau_{\geq 2*}THH(-,\bz_p))\rightarrow R\Gamma_{\mathrm{syn}}(A,\tau_{\geq 2*}HH(-/\bz_p,\bz_p)),$$
which is the desired morphism by definition of the BMS-filtration. Taking the colimit of both sides, we obtain the canonical map $THH(A,\bz_p)\rightarrow HH(A/\bz_p,\bz_p)$, since the BMS-filtration on both $THH(A,\bz_p)$ and $HH(A/\bz_p,\bz_p)$ is exhaustive. 
The proof for $TP$ and $TC^-$ is the same. 

\end{proof}

\subsection{Definition of the motivic filtration}
\begin{defn} Let $A$ be a quasi-lci ring with bounded torsion.  We define $F^*THH(A)$ as the fiber product
\[ \xymatrix{
F^*THH(A)\ar[d]\ar[r]& \mathrm{Fil}_{HKR}^*HH(A) \ar[d]\\
\prod_p\mathrm{Fil}_{BMS}^*THH(A^{\wedge}_p,\bz_p)\ar[r]&\prod_p\mathrm{Fil}_{BMS}^*HH(A^{\wedge}_p/\bz_p,\bz_p) 
}
\]
of $\mathbb{E}_{\infty}$-algebra objects in the symmetric monoidal $\infty$-category $\widehat{DF}(\bs[\bt])$.
\end{defn}

\begin{prop}\label{pcomplTHH}
If $A$ is a quasi-lci ring with bounded torsion, then $F^*THH(A)$ is a  functorial, $\bt$-equivariant, $\bn^{\mathrm{op}}$-indexed, multiplicative, complete, exhaustive filtration on $THH(A)$ endowed with an equivalence
$$(F^*THH(A))^{\wedge}_p\stackrel{\sim}{\longrightarrow} \mathrm{Fil}_{BMS}^*THH(A^{\wedge}_p,\bz_p)$$
for every prime number $p$. 
\end{prop}
\begin{proof}
The fact that $F^*THH(A)$ is $\bn^{\mathrm{op}}$-indexed and exhaustive follows from Proposition \ref{decompoTHH}, since the HKR filtration on $HH(A)$, the BMS filtration on $HH(A^{\wedge}_p/\bz_p,\bz_p)$, and the BMS filtration on $THH(A^{\wedge}_p/\bz_p,\bz_p)$ are all $\bn^{\mathrm{op}}$-indexed and exhaustive. By definition $F^*THH(A)$ is an $\mathbb{E}_{\infty}$-algebra object in the symmetric monoidal $\infty$-category $\widehat{DF}(\bs[\bt])$, hence it is complete, multiplicative and $\bt$-equivariant. Applying $(-)^{\wedge}_p$ to the defining cartesian square of $F^*THH(A)$ we obtain a cartesian square
\[ \xymatrix{
(F^*THH(A))^{\wedge}_p\ar[d]\ar[r]& (\mathrm{Fil}_{HKR}^*HH(A))^{\wedge}_p \ar[d]^{\simeq}\\
\mathrm{Fil}_{BMS}^*THH(A^{\wedge}_p,\bz_p)\ar[r]&\mathrm{Fil}_{BMS}^*HH(A^{\wedge}_p/\bz_p,\bz_p) 
}
\]
where the right vertical map is an equivalence by Proposition \ref{mapfiltHKR}. It follows that the left vertical map is an equivalence as well.

\end{proof}

\begin{defn} Let $A$ be a quasi-lci ring with bounded torsion.  For $?=P,C^-$
we define  $F^*T?(A)$  as the fiber product
\[ \xymatrix{
F^*T?(A)\ar[d]\ar[r]& \mathrm{Fil}_{B}^*H?(A) \ar[d]\\
\prod_p\mathrm{Fil}_{BMS}^*T?(A^{\wedge}_p,\bz_p)\ar[r]&\prod_p\mathrm{Fil}_{BMS}^*H?(A^{\wedge}_p/\bz_p,\bz_p) 
}
\]
of $\mathbb{E}_{\infty}$-algebra objects in the symmetric monoidal $\infty$-category $\widehat{DF}(\bs)$.

There is a morphism $F^*TC^-(A)\rightarrow F^*TP(A)$ of $\mathbb{E}_{\infty}$-algebra objects, and we define $F^*\Sigma^2TC^+(A)\in \widehat{DF}(\bs)$ as the cofiber
$$F^*\Sigma^2TC^+(A):=\mathrm{Cofib}(F^*TC^-(A)\rightarrow F^*TP(A))$$
computed in the stable category $\widehat{DF}(\bs)$.
\end{defn}

\begin{lem}\label{lemnice}
 Let $A$ be a quasi-lci ring with bounded torsion.  For $?=P,C^-$, the canonical maps
$$\mathrm{colim}\, \prod_p\mathrm{Fil}_{BMS}^*T?(A^{\wedge}_p,\bz_p)\rightarrow \prod_pT?(A^{\wedge}_p,\bz_p)$$
$$\mathrm{colim}\, \prod_p\mathrm{Fil}_{BMS}^*H?(A^{\wedge}_p,\bz_p)\rightarrow \prod_pH?(A^{\wedge}_p,\bz_p)$$
are equivalences.
\end{lem}
\begin{proof}We prove the result for $TP$; the same argument works for the other cases. We have $$TP(R,\bz_p)/\mathrm{Fil}^n_{BMS}\simeq \tau_{<2n} TP(R,\bz_p)\in\mathrm{Sp}_{<2n}$$
for any  $R\in \mathrm{QRSPerfd}_{\bz_p}$, hence 
$$\prod_pTP(A^{\wedge}_p,\bz_p)/\mathrm{Fil}^n_{BMS}\simeq \prod_p R\Gamma_{\mathrm{syn}}(A^{\wedge}_p, \tau_{<2n} TP(-,\bz_p)) \in\mathrm{Sp}_{<2n},$$
since $\mathrm{Sp}_{<2n}$ is closed under limits in $\mathrm{Sp}$, see \cite[Corollary 1.2.1.6]{LurieHA}. The map
$$\pi_i(\prod_p \mathrm{Fil}^n_{BMS}TP(A^{\wedge}_p,\bz_p))\rightarrow \pi_i(\prod_p TP(A^{\wedge}_p,\bz_p))$$
is therefore an isomorphism for any $i\geq 2n$, hence the left group is eventually constant when $i\in\bz$ is fixed and $n$ tends to $-\infty$. The result follows since $\pi_i:\mathrm{Sp}\rightarrow \mathrm{Ab}$ commutes with filtered colimits.
\end{proof}

\begin{prop}\label{propmotfiltTP}
Let $A$ be a quasi-lci ring with bounded torsion. 
\begin{enumerate}
\item $F^*TP(A)$ and $F^*TC^-(A)$  are  functorial, multiplicative, complete, exhaustive filtrations on $TP(A)$ and $TC^-(A)$ respectively. 
\item There are canonical equivalences
$$(F^*TP(A))^{\wedge}_p\stackrel{\sim}{\longrightarrow} \mathrm{Fil}_{BMS}^*TP(A^{\wedge}_p,\bz_p)$$
and
$$(F^*TC^-(A))^{\wedge}_p\stackrel{\sim}{\longrightarrow} \mathrm{Fil}_{BMS}^*TC^-(A^{\wedge}_p,\bz_p)$$
for every prime number $p$. 
\item  $F^*\Sigma^2TC^+(A)$  is a  functorial, $\bn_{>0}^{\mathrm{op}}$-indexed, complete, exhaustive filtration on $\Sigma^2TC^+(A)$, and there is an equivalence
$$(F^*\Sigma^2TC^+(A))^{\wedge}_p\stackrel{\sim}{\longrightarrow} \mathrm{Fil}_{BMS}^*\Sigma^2TC^+(A^{\wedge}_p,\bz_p)$$
for every prime number $p$. 
\item For any $n\in\mathbb{Z}$, 
 $\mathrm{gr}^{n}_FTP(A)$, $\mathrm{gr}^{n}_FTC^-(A)$, and $\mathrm{gr}^{n}_F\Sigma^2TC^+(A)$ are $H\bz$-modules and there is a cofiber sequence of $H\bz$-modules 
$$\mathrm{gr}^{n}_FTC^-(A)\rightarrow \mathrm{gr}^{n}_FTP(A)\rightarrow \mathrm{gr}^{n}_F\Sigma^2TC^+(A).$$

\item For any $n\geq 0$, we have a cartesian square of $H\bz$-modules
\[ \xymatrix{
\mathrm{gr}_F^n \Sigma^2TC^+(A)\ar[d]\ar[r]& L\Omega^{<n}_{A/\bz}[2n] \ar[d]\\
\prod_p\widehat{\Prism}_{A^{\wedge}_p}\{n\}/\mathcal{N}^{\geq n}\widehat{\Prism}_{A^{\wedge}_p}\{n\}[2n]\ar[r]&\prod_p (L\Omega^{<n}_{A^{\wedge}_p/\bz_p})^{\wedge}_p[2n] 
}
\] 
where $\widehat{\Prism}_{A^{\wedge}_p}\{n\}/\mathcal{N}^{\geq n}\widehat{\Prism}_{A^{\wedge}_p}\{n\}$ is defined in \cite{BMS}.
\end{enumerate}
\end{prop}

\begin{proof}
(1) We set $?=P,C^-$. The filtration  $F^*T?(A)$ is complete, multiplicative and functorial by definition. It remains to see that it is exhaustive. Applying $\mathrm{colim}$ to the cartesian square of filtrations defining $F^*T?(A)$, we obtain the commutative diagram
\[ \xymatrix{
\mathrm{colim}\,F^*TP(A)\ar[d]\ar[r]& \mathrm{colim}\,\mathrm{Fil}_{B}^*HP(A) \ar[d]\\
\mathrm{colim}\, \prod_p\mathrm{Fil}_{BMS}^*TP(A^{\wedge}_p,\bz_p)\ar[r]\ar[d]^{\simeq}&\mathrm{colim}\, \prod_p\mathrm{Fil}_{BMS}^* HP(A^{\wedge}_p/\bz_p,\bz_p)\ar[d]^{\simeq} \\
 \prod_p TP(A^{\wedge}_p,\bz_p)\ar[r]& \prod_p  HP(A^{\wedge}_p/\bz_p,\bz_p)
}
\]
where the lower vertical maps are equivalences by Lemma \ref{lemnice}. The upper square is obviously cocartesian, hence cartesian. Moreover, 
the map $\mathrm{colim}\,\mathrm{Fil}_{B}^*HP(A)\rightarrow HP(A)$ is an equivalence by \cite{Antieau18} since $A/\bz$ is quasi-lci.  It then follows from  Proposition \ref{propfibproduct} that the canonical map
$$\mathrm{colim}\,F^*TP(A)\rightarrow TP(A)$$
is an equivalence.

(2) Applying $(-)^{\wedge}_p$ to the defining cartesian square for $F^*TP(A)$ we obtain a cartesian square of spectra
\[ \xymatrix{
(F^*TP(A))^{\wedge}_p\ar[d]\ar[r]& (\mathrm{Fil}_{B}^*HP(A))^{\wedge}_p \ar[d]^{\simeq}\\
\mathrm{Fil}_{BMS}^*TP(A^{\wedge}_p,\bz_p)\ar[r]&\mathrm{Fil}_{BMS}^*HP(A^{\wedge}_p/\bz_p,\bz_p) 
}
\]
where the right vertical map is an equivalence by Proposition \ref{prop-FilBpcompleteed}. It follows that the left vertical map is an equivalence as well. The equivalence
$$(F^*TC^-(A))^{\wedge}_p\stackrel{\sim}{\longrightarrow} \mathrm{Fil}_{BMS}^*TC^-(A^{\wedge}_p,\bz_p)$$
is obtained the same way.

(3) We consider the filtration 
$$\mathrm{Fil}_{BMS}^*\Sigma^2TC^+(A^{\wedge}_p,\bz_p):=\mathrm{Cofib}(\mathrm{Fil}_{BMS}^*TC^-(A^{\wedge}_p,\bz_p)\rightarrow \mathrm{Fil}_{BMS}^*TP(A^{\wedge}_p,\bz_p))$$
on $\Sigma^2TC^+(A^{\wedge}_p,\bz_p)$ and we define similarly filtrations $\mathrm{Fil}_{BMS}^*\Sigma^2HC(A^{\wedge}_p,\bz_p)$ and $\mathrm{Fil}_{B}^*\Sigma^2HC(A)$
on $\Sigma^2HC(A^{\wedge}_p,\bz_p)$ and $\Sigma^2HC(A)$ respectively. These three filtrations are $\bn_{>0}^{\mathrm{op}}$-indexed and exhaustive. Indeed, we have equivalences $\mathrm{gr}_B^n HP(A)\simeq L\widehat{\Omega}_{A/\bz}[2n]$, 
$\mathrm{gr}_B^n HC^-(A)\simeq L\widehat{\Omega}^{\geq n}_{A/\bz}[2n]$ and a cofiber sequence
$$\mathrm{gr}_B^n HC^-(A)\rightarrow \mathrm{gr}_B^n HP(A)\rightarrow \mathrm{gr}_B^n \Sigma^2HC(A),$$
hence $\mathrm{Fil}_{B}^*\Sigma^2HC(A)$ is $\bn_{>0}^{\mathrm{op}}$-indexed. It is exhaustive since $\mathrm{Fil}_{B}^*HC^-(A)$ and $\mathrm{Fil}_{B}^*HP(A)$ are exhaustive by \cite{Antieau18}. By \cite[Theorem 1.17]{BMS}, the same argument shows that $\mathrm{Fil}_{BMS}^*\Sigma^2HC(A^{\wedge}_p,\bz_p)$ is $\bn_{>0}^{\mathrm{op}}$-indexed and exhaustive. Similarly, the fact that $\mathrm{Fil}_{BMS}^*\Sigma^2TC^+(A^{\wedge}_p,\bz_p)$ is $\bn_{>0}^{\mathrm{op}}$-indexed and exhaustive follows from \cite[Theorem 1.12]{BMS}.

We have a cartesian square
\[ \xymatrix{
F^*\Sigma^2TC^+(A)\ar[d]\ar[r]& \mathrm{Fil}_{B}^*\Sigma^2HC(A) \ar[d]\\
\prod_p\mathrm{Fil}_{BMS}^*\Sigma^2TC^+(A^{\wedge}_p,\bz_p)\ar[r]&\prod_p\mathrm{Fil}_{BMS}^*\Sigma^2HC(A^{\wedge}_p/\bz_p,\bz_p) 
}
\] 
in the stable category $\widehat{DF}(\bs)$.  It follows that the filtration $F^*\Sigma^2TC^+(A)$ is $\bn_{>0}^{\mathrm{op}}$-indexed since the functor
$\mathrm{gr}^n:\widehat{DF}(\bs)\rightarrow \mathrm{Sp}$
is exact for any $n\in\bz$. Taking the colimit of this diagram we obtain, by the discussion above, the cartesian square
\[ \xymatrix{
F^0\Sigma^2TC^+(A)\ar[d]\ar[r]& \Sigma^2HC(A) \ar[d]\\
\prod_p\Sigma^2TC^+(A^{\wedge}_p,\bz_p)\ar[r]&\prod_p\Sigma^2HC(A^{\wedge}_p/\bz_p,\bz_p) 
}
\] 
Hence the map
$$\mathrm{colim}\,F^*\Sigma^2TC^+(A)\simeq F^0\Sigma^2TC^+(A)\rightarrow \Sigma^2TC^+(A)$$
is an equivalence by Proposition \ref{propfibproduct}.

(4) The morphism $F^*TC^-(A)\rightarrow F^*TP(A)$ is a morphism of complete multiplicative filtrations and the associated graded functor $$(\mathrm{gr}^n)_{n\in\bz}:\widehat{DF}(\bs)\rightarrow\Prod_{\bz}\mathrm{Sp}$$ is symmetric monoidal \cite{Gwilliam-Pavlov18}. It follows
 that $\mathrm{gr}_F^0TC^-(A)\rightarrow \mathrm{gr}_F^0TP(A)$ is a morphism of $\mathbb{E}_{\infty}$-algebras, $\mathrm{gr}_F^nTC^-(A)$ is a $\mathrm{gr}_F^0TC^-(A)$-module, $\mathrm{gr}_F^nTP(A)$ is a $\mathrm{gr}_F^0TP(A)$-module, and the map
 $\mathrm{gr}_F^nTC^-(A)\rightarrow \mathrm{gr}_F^nTP(A)$ is a morphism of $\mathrm{gr}_F^0TC^-(A)$-modules. In view of the cofiber sequence of spectra
 $$\mathrm{gr}^{n}_FTC^-(A)\rightarrow \mathrm{gr}^{n}_FTP(A)\rightarrow \mathrm{gr}^{n}_F\Sigma^2TC^+(A)$$
one is reduced to showing that $\mathrm{gr}_F^0TC^-(A)$ is an $\mathbb{E}_{\infty}$-$\bz$-algebra. Since $\mathrm{gr}_F^0TC^-(A)$ is a $\mathbb{E}_{\infty}$-$\mathrm{gr}_F^0TC^-(\bz)$-algebra, it is enough to show that $\mathrm{gr}_F^0TC^-(\bz)$ is an $\mathbb{E}_{\infty}$-$\bz$-algebra. But 
\[ \xymatrix{
\mathrm{gr}_F^0TC^-(\bz)\ar[d]\ar[r]& \mathrm{gr}_B^0HC^-(\bz)\simeq \bz \ar[d]\\
\prod_p\mathrm{gr}_{BMS}^0TC^-(\bz_p,\bz_p)\ar[r]&\prod_p \mathrm{gr}_{BMS}^0HC^-(\bz_p/\bz_p,\bz_p) \simeq \prod_p \bz_p
}
\] 
is a cartesian square of $\mathbb{E}_{\infty}$-algebras in spectra. Moreover, $\mathrm{gr}_{BMS}^0TC^-(\bz_p,\bz_p)$ is an $\mathbb{E}_{\infty}$-$\bz_p$-algebra and the map 
\begin{equation}\label{Zpmap}
\mathrm{gr}_{BMS}^0TC^-(\bz_p,\bz_p)\rightarrow \mathrm{gr}_{BMS}^0HC^-(\bz_p,\bz_p)
\end{equation}
is a morphism of $\mathbb{E}_{\infty}$-$\bz_p$-algebras. Indeed, we have 
$\mathrm{gr}_{BMS}^0TC^-(\bz_p,\bz_p)\simeq R\Gamma_{\mathrm{syn}}(\bz_p,\pi_{0} TC^-(-,\bz_p))$ and $\mathrm{gr}_{BMS}^0HC^-(\bz_p,\bz_p)\simeq R\Gamma_{\mathrm{syn}}(\bz_p,\pi_{0} HC^-(-,\bz_p))$ where $\pi_{0} TC^-(-,\bz_p)$ and $\pi_{0} HC^-(-,\bz_p)$ are sheaves of $\mathbb{E}_{\infty}$-$\bz_p$-algebras on $\mathrm{qrsPerfd}_{\bz_p}^{\mathrm{op}}$ and the map  (\ref{Zpmap}) is induced by the canonical morphism of sheaves of $\mathbb{E}_{\infty}$-$\bz_p$-algebras
$\pi_{0} TC^-(-,\bz_p)\rightarrow \pi_{0} HC^-(-,\bz_p)$. Since the equivalence (\ref{unfolding}) is symmetric monoidal, it follows that (\ref{Zpmap}) is a morphism of $\mathbb{E}_{\infty}$-$\bz_p$-algebras.

Moreover, the map 
$$\bz=\mathrm{gr}_B^0HC^-(\bz)\rightarrow \mathrm{gr}_{BMS}^0HC^-(\bz_p/\bz_p,\bz_p) \simeq \bz_p$$
is the $p$-completion map, i.e. the unique morphism of rings. We obtain a commutative square
 \[ \xymatrix{
H\bz\ar[d]\ar[r]& \mathrm{gr}_B^0HC^-(\bz) \ar[d]\\
\prod_p\mathrm{gr}_{BMS}^0TC^-(\bz_p,\bz_p)\ar[r]&\prod_p \mathrm{gr}_{BMS}^0HC^-(\bz_p/\bz_p,\bz_p) 
}
\] 
which gives a morphism of $\mathbb{E}_{\infty}$-algebras $H\bz\rightarrow \mathrm{gr}_F^0TC^-(\bz)$.

(5) For any $n\geq 0$, we have a cartesian square of $H\bz$-modules
\[ \xymatrix{
\mathrm{gr}_F^n \Sigma^2TC^+(A)\ar[d]\ar[r]& \mathrm{gr}_B^{n}\Sigma^2HC(A) \ar[d]\\
\prod_p\mathrm{gr}_{BMS}^n \Sigma^2TC^+(A^{\wedge}_p,\bz_p)\ar[r]&\prod_p \mathrm{gr}_{BMS}^n \Sigma^2HC(A^{\wedge}_p/\bz_p,\bz_p) 
}
\] 
hence (5) follows from the identification of $\mathrm{gr}_{BMS}^n$ and $\mathrm{gr}_{B}^n$ given in \cite{BMS} and \cite{Antieau18} respectively.

\end{proof}

\section{The bifiltration $\mathcal{Z}_{\bz}^*F^*$}\label{SectionBifilt}

The goal of this section is to prove Theorem \ref{thm-intro}. We first prove the analogous result for $p$-completely quasisyntomic rings in Sections \ref{sectfirstproof} and \ref{sectbifiltlocal}, where we define a bifiltration $\mathcal{Z}_{\bz_p}^*\mathrm{Fil}_{BMS}^*$ which interpolates between the BMS-filtration and the filtration of Proposition \ref{prop-filt-Z}. We give a direct construction of the induced filtration $\mathcal{Z}^*\mathrm{gr}^n_{BMS}THH(A,\bz_p)$ in Section \ref{sectfirstproof}, which works for  $p$-completely flat $\bz_p$-algebras. A lengthier construction is then given in Section \ref{sectbifiltlocal}, and is obtained as follows: we localize the base ring $\bz_p$ for the quasisyntomic topology, define our bifiltration locally and apply quasisyntomic descent. This gives the bifiltration $\mathcal{Z}_{\bz_p}^*\mathrm{Fil}_{BMS}^*$ and makes transparent its completeness and multiplicative properties. Moreover, this second construction works in full generality, i.e. for arbitrary quasisyntomic rings, and also applies to more general base rings (see Theorem \ref{thmpadic}). On the other hand, it requires more work, and the direct construction of Section \ref{sectfirstproof} is still necessary to compute $\mathrm{gr}^n_{BMS}THH(\bz_p,\bz_p)$.

\subsection{The filtration on $\mathrm{gr}^n_{BMS}THH(-,\bz_p)$}\label{sectfirstproof}

\begin{prop}\label{thmkey}
For any $A\in\mathrm{qSyn}_{\bz_p}$ and any $n\geq 0$, there is a   functorial, decreasing, $\bn^{\mathrm{op}}$-indexed, finite,  complete and  exhaustive filtration $\mathcal{Z}^*\mathrm{gr}^n_{BMS}THH(A,\bz_p)$ on the $\bz_p$-complex $\mathrm{gr}^n_{BMS}THH(A,\bz_p)$ with graded pieces
\begin{eqnarray*}
\mathrm{gr}^0_{\mathcal{Z}}\mathrm{gr}^n_{BMS}THH(A,\bz_p)&\simeq& (L\Lambda^{n}L_{A/\bz_p})^{\wedge}_p[n];\\
\mathrm{gr}^j_{\mathcal{Z}}\mathrm{gr}^n_{BMS}THH(A,\bz_p)&\simeq &L\Lambda^{n-j}L_{A/\bz_p}\otimes^L_{\bz_p}\bz_p/j\bz_p[n+j-1]\textrm{ for }j\geq 1.
\end{eqnarray*}
\end{prop}
\begin{proof}
We first treat the case $n=0$. For any $S\in \mathrm{qrsPerfd}_{\bz_p}$, we have $$\mathrm{Fil}^*_{BMS}THH(S,\bz_p)\simeq \tau_{\geq 2*}THH(S,\bz_p)$$ hence 
$$\mathrm{gr}^0_{BMS}THH(S,\bz_p)\simeq \pi_{0}THH(S,\bz_p)\simeq S$$ 
since $\pi_{*}THH(S,\bz_p)$ is concentrated in even degrees by \cite[Theorem 7.1]{BMS}. For arbitrary $A\in\mathrm{qSyn}_{\bz_p}$, we obtain
$$\mathrm{gr}^0_{BMS}THH(A,\bz_p):=R\Gamma_{\mathrm{syn}}(A,\mathrm{gr}^0_{BMS}THH(-,\bz_p))\simeq A$$
by faithfully flat descent. So the result is proven for $n=0$ and we may suppose from now on that $n\geq 1$.

Let $S\in \mathrm{qrsPerfd}_{\bz_p}$. We consider the canonical map
$$THH(S,\bz_p)\rightarrow HH(S/\bz_p,\bz_p)$$
and we denote by $T(S,\bz_p)=\mathcal{Z}^1THH(S,\bz_p)[1]$ its cofiber. By Proposition \ref{prop-filt-Z},  $T(S,\bz_p)$ has an $\bn_{>0}^{\mathrm{op}}$-indexed complete exhaustive decreasing filtration $Z^*T(S,\bz_p)$ with $i$-th graded piece
$$\mathrm{gr}_Z^{i}T(S,\bz_p)\simeq HH (S/\bz_p,\bz_p)\otimes^L_{\bz}\bz/i\bz[2i]$$
for any $i\geq 1$. Moreover, the $p$-completion of the HKR filtration induces an $\bn^{\mathrm{op}}$-indexed complete exhaustive decreasing  filtration on $HH (S/\bz_p,\bz_p)\otimes^L_{\bz}\bz/i\bz[2i]$ with $j$-th graded piece
\begin{eqnarray*}
\left(\mathrm{gr}_{HKR}^jHH(S/\bz_p,\bz_p)\right)\otimes^L_{\bz}\bz/i\bz[2i]&\simeq &(L\Lambda^j L_{S/\bz_p})^{\wedge}_p[j] \otimes^L_{\bz}\bz/i\bz[2i]
\end{eqnarray*}
for any $j\geq 0$. Since $L\Lambda^j L_{S/\bz_p}[-j]\in D(S)$ is $p$-completely flat by \cite[Lemma 5.14]{BMS}, and since $\bz_p\rightarrow S$ is $p$-completely flat, it follows that 
$L\Lambda^j L_{S/\bz_p}[-j]\in D(\bz_p)$ is $p$-completely flat by Lemma \ref{trivlemma}. Hence the complex
$$(L\Lambda^j L_{S/\bz_p})^{\wedge}_p \otimes^L_{\bz_p}\bz_p/i\bz_p[2i+j]$$ 
is concentrated in homological degree $2(i+j)$. Let $i\geq 1$. By induction and completeness of the HKR filtration, it follows that  
$\pi_*\left(HH (S/\bz_p,\bz_p)\otimes^L_{\bz}\bz/i\bz[2i]\right)$ is concentrated in even homological degrees $\geq 2i$. By induction and completeness of the filtration $Z^*T(S,\bz_p)$, it follows that $\pi_*(Z^{i}T(S,\bz_p))$ is concentrated in even homological degrees $\geq 2i$ as well. 

In particular $\pi_{*} T(S,\bz_p)=\pi_*(Z^{1}T(S,\bz_p))$ is concentrated in even degrees.  Moreover $\pi_{*} HH(S,\bz_p)$ and $\pi_{*} THH(S,\bz_p)$ are both concentrated in even degrees by \cite[Lemma 5.14]{BMS} and \cite[Theorem 7.1]{BMS} respectively. We obtain a short exact sequence
$$0\rightarrow \pi_{2n} THH(S,\bz_p)\rightarrow \pi_{2n} HH(S/\bz_p,\bz_p) \rightarrow \pi_{2n} T(S/\bz_p,\bz_p)\rightarrow 0$$
which we may rewrite as follows:
\begin{equation}\label{firststep}
0\rightarrow \pi_{2n} THH(S,\bz_p)\rightarrow (L\Lambda^{n}L_{S/\bz_p})^{\wedge}_p[-n] \rightarrow \pi_{2n} T(S/\bz_p,\bz_p)\rightarrow 0.
\end{equation}
For any $i\geq 1$, the fiber sequence
$$Z^{i+1}T(S,\bz_p) \rightarrow Z^{i}T(S,\bz_p) \rightarrow  \mathrm{gr}^{i}_ZT(S,\bz_p)$$
induces an exact sequence
$$0\rightarrow \pi_{2n}(Z^{i+1}T(S,\bz_p)) \rightarrow \pi_{2n}(Z^{i}T(S,\bz_p)) \rightarrow  \pi_{2n}(\mathrm{gr}_Z^{i}T(S,\bz_p))\rightarrow 0$$
and we have
\begin{eqnarray*}
 \pi_{2n}(\mathrm{gr}_Z^{i}T(S,\bz_p))&\simeq &\pi_{2n} \left(HH (S/\bz_p,\bz_p)\otimes^L_{\bz}\bz/i\bz[2i]\right)\\
 &\simeq&\left( \pi_{2(n-i)} HH (S/\bz_p,\bz_p)\right)\otimes_{\bz_p}\bz_p/i\bz_p\\
 &\simeq& (L\Lambda^{n-i}L_{S/\bz_p}[i-n])^{\wedge}_p\otimes_{\bz_p}\bz_p/i\bz_p.
 \end{eqnarray*}
Hence we have a finite decreasing filtration
 $$\pi_{2n}(Z^{1}T(S,\bz_p))\hookleftarrow\cdots \hookleftarrow \pi_{2n}(Z^{n}T(S,\bz_p))\hookleftarrow  \pi_{2n}(Z^{n+1}T(S,\bz_p))=0$$
in the classical (i.e. underived) sense on the module
$$\pi_{2n}(T(S,\bz_p))=\pi_{2n}(Z^{1}T(S,\bz_p))$$
with graded pieces
\begin{equation}\label{afterstep}
\pi_{2n}(Z^{i}T(S,\bz_p))/\pi_{2n}(Z^{i+1}T(S,\bz_p))\simeq (L\Lambda^{n-i}L_{S/\bz_p}[i-n])^{\wedge}_p\otimes_{\bz_p}\bz_p/i\bz_p.
\end{equation}
We consider the (derived) filtration on $\pi_{2n} THH(S,\bz_p)$ given by 
\begin{eqnarray*}
Z^0\pi_{2n} THH(S,\bz_p)&:=&\pi_{2n} THH(S,\bz_p);\\ 
Z^i\pi_{2n} THH(S,\bz_p)&:=&\pi_{2n}(Z^{i}T(S,\bz_p))[-1]\textrm{ for }i\geq 1;
\end{eqnarray*}
where the map $$Z^1\pi_{2n} THH(S,\bz_p)=\pi_{2n}(T(S,\bz_p))[-1]\rightarrow \pi_{2n} THH(S,\bz_p)=Z^0\pi_{2n} THH(S,\bz_p)$$ is (the shift of) the boundary map induced by (\ref{firststep}). Since $S\in\mathrm{qrsPerfd}_{\bz_p}$, we have
$$\mathrm{gr}^n_{BMS}THH(S,\bz_p)\simeq \pi_{2n} THH(S,\bz_p)[2n]$$
and the filtration
$$\mathcal{Z}^{*} \mathrm{gr}^n_{BMS}THH(S,\bz_p):=(Z^i\pi_{2n} THH(S,\bz_p))[2n]$$
satisfies the conclusion of the proposition, by (\ref{firststep}) and (\ref{afterstep}). 

For arbitrary $A\in\mathrm{qSyn}_{\bz_p}$, we define
$$\mathcal{Z}^{*} \mathrm{gr}^n_{BMS}THH(A,\bz_p):=R\Gamma_{\mathrm{syn}}(A, \mathcal{Z}^{*} \mathrm{gr}^n_{BMS}THH(-,\bz_p)).$$
The conclusion of the theorem remains true for any such $A$ since $(L\Lambda^{n-j}L_{-/\bz_p})^{\wedge}_p$ is a quasisyntomic sheaf by \cite[Theorem 3.1]{BMS}.

\end{proof}

\begin{lem}\label{trivlemma}
Let $A\rightarrow B$ be a $p$-completely flat map of $p$-complete rings with bounded $p^{\infty}$-torsion. If $M\in D(B)$ is $p$-completely flat over $B$, then $M\in D(A)$ is $p$-completely flat over $A$.  
\end{lem}
\begin{proof}
Since $B\in D(A)$ is  $p$-completely flat we have $B\otimes^L_A A/pA\simeq B/pB$.
We obtain that
$$M\otimes^L_AA/pA\simeq M\otimes^L_BB\otimes^L_AA/pA\simeq M\otimes^L_BB/pB$$
is concentrated in degree $0$ such that $H^0(M\otimes^L_BB/pB)$ is a flat $B/pB$-module. Since $A/pA\rightarrow B/pB$ is also flat,   $H^0(M\otimes^L_BB/pB)$ is a flat $A/pA$-module as well.
\end{proof}

\begin{cor}\label{key}
There is a canonical equivalence
$$\mathrm{Fil}^*_{BMS}THH(\bz_p,\bz_p)\stackrel{\sim}{\rightarrow} \tau_{\geq 2*-1} THH(\bz_p,\bz_p).$$
In particular, we have $\mathrm{gr}^0_{BMS}THH(\bz_p,\bz_p)\simeq\bz_p$ and $$\mathrm{gr}^n_{BMS}THH(\bz_p,\bz_p)\simeq\bz_p/n\bz_p[2n-1]\textrm{ for }n\geq 1.$$
Finally we have
$\mathrm{gr}^0_{F}THH(\bz)\simeq\bz$ and $$\mathrm{gr}^n_{F}THH(\bz)\simeq\bz/n\bz[2n-1]\textrm{ for }n\geq 1.$$

\end{cor}

\begin{proof}
The second assertion follows immediately from Proposition \ref{thmkey} since $L\Lambda^{n-j}L_{\bz_p/\bz_p}=\bz_p$ for $n=j$ and $L\Lambda^{n-j}L_{\bz_p/\bz_p}=0$ otherwise. The third assertion then follows from the definition of $\mathrm{gr}^n_{F}THH(\bz)$.

In particular, $\mathrm{gr}^n_{BMS}THH(\bz_p,\bz_p)$ is concentrated in homological degrees $\geq 2n-1$ for any $n\geq 0$. It follows that $\pi_*(\mathrm{Fil}^n_{BMS}THH(\bz_p,\bz_p))$ is concentrated in homological degrees $\geq 2n-1$. Hence  the map $\mathrm{Fil}^n_{BMS}THH(\bz_p,\bz_p)\rightarrow THH(\bz_p,\bz_p)$
factors through $\tau_{\geq 2n-1}THH(\bz_p,\bz_p)$. We obtain a morphism of complete filtrations
$$\mathrm{Fil}^*_{BMS}THH(\bz_p,\bz_p)\rightarrow\tau_{\geq 2*-1} THH(\bz_p,\bz_p)$$
which induces an equivalence on graded pieces.

\end{proof}

\subsection{The bifiltration for $p$-adic rings}\label{sectbifiltlocal}

\begin{prop}\label{prop1}
Let $R\in \mathrm{QSyn}_{\bz_p}$. Let $A,S\in \mathrm{qrsPerfd}_{R}$ and suppose that $R\rightarrow S$ is a quasisyntomic cover. If $R$ is perfectoid or $R=\bz_p$, we allow more generally $A\in \mathrm{QRSPerfd}_{R}$.

\begin{enumerate}
\item There is an $\bn^{\mathrm{op}}$-indexed decreasing multiplicative complete exhaustive $\bt$-equivariant filtration 
$\mathcal{Z}_S^*\mathrm{Fil}^*_{BMS} THH(A\widehat{\otimes}_R S,\bz_p)$
on $\mathrm{Fil}^*_{BMS}THH(A\widehat{\otimes}_R S,\bz_p)$
 endowed with equivalences
$$\mathrm{gr}^j_{\mathcal{Z}_S} \mathrm{Fil}^{*}_{BMS}THH(A\widehat{\otimes}_R S,\bz_p)\simeq \left( \mathrm{Fil}^{*-j}_{BMS}HH(A/R,\bz_p)\right)\widehat{\otimes}_{R}(\pi_{2j}THH(S,\bz_p))[2j]$$
for any $j\in \bz$.

\item For $?=P,C^-$, there is an  $\bn^{\mathrm{op}}$-indexed decreasing multiplicative complete exhaustive filtration
$\mathcal{Z}_S^*\mathrm{Fil}^*_{BMS} T?(A\widehat{\otimes}_R S,\bz_p)$
on $\mathrm{Fil}^*_{BMS}T?(A\widehat{\otimes}_R S,\bz_p)$ 
endowed with canonical maps
$$ \left( \mathrm{Fil}^{*-j}_{BMS}H?(A/R,\bz_p)\right)\widehat{\otimes}_{R}(\pi_{2j}THH(S,\bz_p))[2j]\rightarrow  \mathrm{gr}^j_{\mathcal{Z}_S} \mathrm{Fil}^{*}_{BMS}T?(A\widehat{\otimes}_R S,\bz_p)$$
inducing equivalences
$$\mathrm{gr}^j_{\mathcal{Z}_S} \mathrm{gr}^{n}_{BMS}TP(A\widehat{\otimes}_R S,\bz_p)\simeq R\mathrm{lim}_k \left((L\Omega^{<k}_{A/R})^{\wedge}_p[2(n-j)] \widehat{\otimes}_{R}(\pi_{2j}THH(S,\bz_p))[2j]\right);$$
$$\mathrm{gr}^j_{\mathcal{Z}_S} \mathrm{gr}^{n}_{BMS}TC^-(A\widehat{\otimes}_R S,\bz_p)\simeq R\mathrm{lim}_k \left((L\Omega^{k> *\geq n-j}_{A/R})^{\wedge}_p[2(n-j)] \widehat{\otimes}_{R}(\pi_{2j}THH(S,\bz_p))[2j]\right)$$
for any $j\in \bz$. Here, for $k\geq n-j$, we denote by $L\Omega^{k> *\geq n-j}_{A/R}$ the cofiber of the map $L\widehat{\Omega}^{\geq k}_{A/R}\rightarrow L\widehat{\Omega}^{\geq n-j}_{A/R}$.

\end{enumerate}

\end{prop}
\begin{proof}  \textbf{(1)} We have  $A,S\in \mathrm{QRSPerfd}_{R}$ and $R\rightarrow S$ is a quasisyntomic cover by assumption. It follows\footnote{It follows from \cite[Lemmas 4.15, 4.16]{BMS} that $D=A\widehat{\otimes}_RS\in \mathrm{QSyn}_{R}$. Moreover $D$ receives a map from a perfectoid ring, as $S$ does. By \cite[Corollary 7.2.1.23]{LurieHA} we have $D/p=\pi_0(A\otimes^L_RS\otimes^L_{\bz_p}\bz_p/p)\simeq A/p\otimes_{R/p} S/p$, which is semi-perfect.} that $A\widehat{\otimes}_RS\in \mathrm{QRSPerfd}_{R}$. We define an $\bn^{\mathrm{op}}$-indexed decreasing exhaustive filtration on $THH(A\widehat{\otimes}_R S,\bz_p)$ as follows:
$$\mathcal{Z}_S^*THH(A\widehat{\otimes}_R S,\bz_p):= THH(A\widehat{\otimes}_R S,\bz_p)\widehat{\otimes}_{THH(S,\bz_p)}\tau_{\geq 2*} THH(S,\bz_p).$$
The filtration $\tau_{\geq 2*} THH(S,\bz_p)\in DF(THH(S,\bz_p))$ is $\bt$-equivariant and multiplicative, hence so is the filtration $\mathcal{Z}_S^*THH(A\widehat{\otimes}_R S,\bz_p)$.
The functor
$$THH(A\widehat{\otimes}_R S,\bz_p)\otimes_{THH(S,\bz_p)}(-):D(THH(S,\bz_p))\longrightarrow \mathrm{Sp}$$
is right $t$-exact \cite[Definition 1.3.3.1]{LurieHA} by \cite[Corollary 7.2.1.23]{LurieHA}, and so is the  $p$-completion functor by \cite[Proposition 7.3.4.4]{LurieSAG}. It follows that the connectivity of  $\mathcal{Z}_S^nTHH(A\widehat{\otimes}_R S,\bz_p)$ tends to $\infty$ as $n\rightarrow \infty$, so that $\mathcal{Z}_S^*THH(A\widehat{\otimes}_R S,\bz_p)$ is also complete.

By \cite[Theorem 7.1(1)]{BMS}, $\pi_*THH(S,\bz_p)$ is concentrated in even degrees. Hence for any $j\geq 0$, we have
\begin{eqnarray}
&&\mathrm{gr}^j_{\mathcal{Z}_S} THH(A\widehat{\otimes}_R S,\bz_p)\\
&\simeq& THH(A\widehat{\otimes}_R S,\bz_p)\widehat{\otimes}_{THH(S,\bz_p)} \pi_{2j} THH(S,\bz_p)[2j] \\
&\simeq& THH(A\widehat{\otimes}_R S,\bz_p)\widehat{\otimes}_{THH(S,\bz_p)}S\widehat{\otimes}_S \pi_{2j} THH(S,\bz_p)[2j]\\
\label{thirdy}&\simeq& (THH(A\otimes^L_R S)\otimes_{THH(S)}S)^{\wedge}_p\widehat{\otimes}_S \pi_{2j} THH(S,\bz_p)[2j]\\
\label{fourthy}&\simeq& THH(A\otimes^L_R S/S,\bz_p)\widehat{\otimes}_S \pi_{2j} THH(S,\bz_p)[2j]\\
\label{fifthy}&\simeq& HH(A/R,\bz_p)\widehat{\otimes}_R S\widehat{\otimes}_S \pi_{2j} THH(S,\bz_p)[2j]\\
\label{lasty}&\simeq& HH(A/R,\bz_p)\widehat{\otimes}_R \pi_{2j} THH(S,\bz_p)[2j].
\end{eqnarray}
Here (\ref{thirdy}) follows from Remark \ref{rem-p-general} and from the fact that the $p$-adic completion functor is monoidal; (\ref{fourthy}) and (\ref{fifthy}) follow from the universal properties of $THH$ and $HH$, see (\ref{computeTHH1}) and (\ref{computeTHH2}).

By \cite[Lemma 5.14, Theorem 7.1]{BMS}, $HH(A/R,\bz_p)$ is concentrated in even degrees, $\pi_{2k} HH(A/R,\bz_p)$ is a $p$-completely flat $A$-module for all $k$, and $\pi_{2l} THH(S,\bz_p)$ is a $p$-completely flat $S$-module for all $l$.  It then follows 
from (\ref{lasty}), completeness of the filtration $(\tau_{\geq 2*}HH(A/R,\bz_p))\widehat{\otimes}_R \pi_{2j} THH(S,\bz_p)[2j]$ and Lemma \ref{lemflat} 
that $\pi_*\mathrm{gr}_{\mathcal{Z}_S}^j THH(A\widehat{\otimes}_R S)$ is concentrated in even degrees. By induction and completeness of the filtration $\mathcal{Z}_S^*$, 
$\pi_*\mathcal{Z}_S^j THH(A\widehat{\otimes}_R S)$ is concentrated in even degrees as well, for any $j\geq 0$.

We now define a bicomplete multiplicative bifiltration on $THH(A\widehat{\otimes}_R S,\bz_p)$ as follows:
$$\mathcal{Z}_S^j\mathrm{Fil}^i_{BMS}THH(A\widehat{\otimes}_R S,\bz_p):=\tau_{\geq 2i} (\mathcal{Z}_S^j THH(A\widehat{\otimes}_R S)).$$
Since $A\widehat{\otimes}_RS\in \mathrm{QRSPerfd}_{R}$, we have
$$\mathrm{Fil}^{*}_{BMS}THH(A\widehat{\otimes}_R S)\simeq \tau_{\geq 2*} THH(A\widehat{\otimes}_R S)\simeq \mathcal{Z}_S^0\mathrm{Fil}^*_{BMS}THH(A\widehat{\otimes}_R S,\bz_p)$$
so that $\mathcal{Z}_S^*\mathrm{Fil}^*_{BMS}THH(A\widehat{\otimes}_R S,\bz_p)$ is an $\bn^{\mathrm{op}}$-indexed decreasing multiplicative complete exhaustive filtration on $\mathrm{Fil}^*_{BMS}THH(A\widehat{\otimes}_R S,\bz_p)$. 

By definition $\mathrm{gr}^j_{\mathcal{Z}_S}\mathrm{Fil}^i_{BMS}THH(A\widehat{\otimes}_R S,\bz_p)$ is the cofiber of the map
$$\tau_{\geq 2i} (\mathcal{Z}_S^{j+1} THH(A\widehat{\otimes}_R S))\rightarrow \tau_{\geq 2i} (\mathcal{Z}_S^j THH(A\widehat{\otimes}_R S)).$$
The fact that $\pi_*\mathcal{Z}_S^{j+1} THH(A\widehat{\otimes}_R S)$ and $\pi_*\mathcal{Z}_S^{j} THH(A\widehat{\otimes}_R S)$ are both concentrated in even degrees therefore gives 
\begin{eqnarray*}
&&\mathrm{gr}^j_{\mathcal{Z}_S}\mathrm{Fil}^i_{BMS}THH(A\widehat{\otimes}_R S,\bz_p)\\
&\simeq& \tau_{\geq 2i}\left(\mathrm{gr}^j_{\mathcal{Z}_S} THH(A\widehat{\otimes}_R S,\bz_p)\right)\\
&\simeq & \tau_{\geq 2i}\left(HH(A/R,\bz_p)\widehat{\otimes}_R \pi_{2j} THH(S,\bz_p)[2j]\right)\\
&\simeq & \left(\tau_{\geq 2(i-j)}HH(A/R,\bz_p)\right)\widehat{\otimes}_R \pi_{2j} THH(S,\bz_p)[2j]\\
&\simeq & \left(\mathrm{Fil}^{i-j}_{BMS}HH(A/R,\bz_p)\right)\widehat{\otimes}_R \pi_{2j} THH(S,\bz_p)[2j]
\end{eqnarray*}
where the last equivalence follows from the fact that $A\in \mathrm{QRSPerfd}_{R}$.
Using \cite[Lemma 5.14(2)]{BMS}, we obtain
$$\mathrm{gr}^j_{\mathcal{Z}_S}\mathrm{gr}^i_{BMS}THH(A\widehat{\otimes}_R S,\bz_p)\simeq (L\Lambda^{i-j}L_{A/R})^{\wedge}_p\widehat{\otimes}_R \pi_{2j} THH(S,\bz_p)[i+j].$$

\textbf{(2)} We give a detailed proof for $TP$; the same proof works for $TC^-$ as well.  We define the following $\bn^{\mathrm{op}}$-indexed decreasing exhaustive filtration
$$\mathcal{Z}_S^*TP(A\widehat{\otimes}_R S,\bz_p):= \left(\mathcal{Z}_S^*THH(A\widehat{\otimes}_R S,\bz_p)\right)^{t\bt}$$
 on $TP(A\widehat{\otimes}_R S,\bz_p)$, which is also multiplicative since $(-)^{t\bt}$ is lax symmetric monoidal.
We define similarly $\bn^{\mathrm{op}}$-indexed decreasing exhaustive filtrations $\mathcal{Z}_S^*TC^-(A\widehat{\otimes}_R S,\bz_p)$ and $\mathcal{Z}_S^*TC^+(A\widehat{\otimes}_R S,\bz_p)$ on $TC^-(A\widehat{\otimes}_R S,\bz_p)$ and $TC^+(A\widehat{\otimes}_R S,\bz_p)$ respectively.
The filtration $\mathcal{Z}_S^*THH(A\widehat{\otimes}_R S,\bz_p)$ is complete, hence so is $\mathcal{Z}_S^*TC^-(A\widehat{\otimes}_R S,\bz_p)$. Since $(-)_{h\bt}$ is right $t$-exact, the filtration $\mathcal{Z}_S^*TC^+(A\widehat{\otimes}_R S,\bz_p)$ is complete as well. Hence $\mathcal{Z}_S^*TP(A\widehat{\otimes}_R S,\bz_p)$ is complete.
For any $j\geq 0$, we have
\begin{eqnarray*}
\mathrm{gr}^j_{\mathcal{Z}_S} TP(A\widehat{\otimes}_R S,\bz_p) &\simeq&  \left(\mathrm{gr}^j_{\mathcal{Z}_S} THH(A\widehat{\otimes}_R S,\bz_p) \right)^{t\bt}\\
\label{sed}&\simeq& \left(HH(A/R,\bz_p)\widehat{\otimes}_R \pi_{2j} THH(S,\bz_p) [2j]\right)^{t\bt} 
\end{eqnarray*}
since the equivalence (\ref{lasty}) is $\bt$-equivariant, where $\bt$ acts trivially on $\pi_{2j} THH(S,\bz_p)$. As in the first part of the proof, it follows that $\pi_*\mathcal{Z}^j_{S} TP(A\widehat{\otimes}_R S,\bz_p)$ is concentrated in even degrees for any $j\geq 0$.

We define a decreasing multiplicative bicomplete biexhaustive bifiltration on $TP(A\widehat{\otimes}_R S,\bz_p)$ as follows:
$$\mathcal{Z}_S^j\mathrm{Fil}^i_{BMS} TP(A\widehat{\otimes}_R S,\bz_p):=\tau_{\geq 2i} (\mathcal{Z}_S^j TP(A\widehat{\otimes}_R S)).$$
Note that, since $A\widehat{\otimes}_RS\in \mathrm{QRSPerfd}_{R}$, we have
$$\mathrm{Fil}^{*}_{BMS}TP(A\widehat{\otimes}_R S)\simeq \tau_{\geq 2*} TP(A\widehat{\otimes}_R S)\simeq \mathcal{Z}_S^0\mathrm{Fil}^*_{BMS} TP(A\widehat{\otimes}_R S,\bz_p).$$ 
Since $\pi_*\mathcal{Z}^j_{S} TP(A\widehat{\otimes}_R S,\bz_p)$ is concentrated in even degrees for any $j\geq 0$, we have  an equivalence
\begin{eqnarray*}
\mathrm{gr}_{\mathcal{Z}_S}^j\mathrm{Fil}^*_{BMS}TP(A\widehat{\otimes}_R S,\bz_p)&\simeq& \tau_{\geq 2*}\left(\mathrm{gr}^j_{\mathcal{Z}_S} TP(A\widehat{\otimes}_R S,\bz_p)\right)\\
&\simeq & \tau_{\geq 2*}\left(HH(A/R,\bz_p)\widehat{\otimes}_R \pi_{2j} THH(S,\bz_p)[2j]\right)^{t\bt}
\end{eqnarray*}
Consider the canonical map 
\begin{equation}\label{multimap}
HP(A/R,\bz_p)\widehat{\otimes}_R \pi_{2j} THH(S,\bz_p)[2j]\rightarrow \left(HH(A/R,\bz_p)\widehat{\otimes}_R \pi_{2j} THH(S,\bz_p)[2j]\right)^{t\bt}
\end{equation}
induced by the fact that $(-)^{t\bt}$ is lax symmetric monoidal and by the canonical map 
$$\pi_{2j} THH(S,\bz_p)\rightarrow (\pi_{2j} THH(S,\bz_p))^{h\bt}\rightarrow (\pi_{2j} THH(S,\bz_p))^{t\bt},$$ 
which is well defined since $\bt$ acts trivially on $\pi_{2j} THH(S,\bz_p)$. 
The complex $$\left(\tau_{\geq 2(i-j)}HP(A/R,\bz_p)\right)\widehat{\otimes}_R \pi_{2j} THH(S,\bz_p)[2j]$$ is concentrated in homological degrees $\geq 2i$ by \cite[Corollary 7.1.2.23]{LurieHA} and \cite[Proposition 7.3.4.4]{LurieSAG}.
Hence  (\ref{multimap}) induces a morphism of filtrations
\begin{eqnarray}
&&\left(\mathrm{Fil}^{*-j}_{BMS}HP(A/R,\bz_p)\right)\widehat{\otimes}_R \pi_{2j} THH(S,\bz_p)[2j]\\
\label{eqiqi}&\simeq &\left(\tau_{\geq 2(*-j)}HP(A/R,\bz_p)\right)\widehat{\otimes}_R \pi_{2j} THH(S,\bz_p)[2j]\\
\label{thatmap}&\rightarrow& \tau_{\geq 2*}\left(HH(A/R,\bz_p)\widehat{\otimes}_R \pi_{2j} THH(S,\bz_p)[2j]\right)^{t\bt}\\
&\simeq & \mathrm{gr}_{\mathcal{Z}_S}^j\mathrm{Fil}^*_{BMS}TP(A\widehat{\otimes}_R S,\bz_p).
\end{eqnarray}
Here (\ref{eqiqi}) holds since $A\in \mathrm{QRSPerfd}_{R}$. Note also  that $HH(A/R,\bz_p)\widehat{\otimes}_R \pi_{2j} THH(S,\bz_p)$ is concentrated in even degrees by Lemma \ref{lemflat}.

We have 
\begin{eqnarray*}
\mathrm{gr}_{\mathcal{Z}_S}^j\mathrm{gr}^{i}_{BMS}TP(A\widehat{\otimes}_R S,\bz_p)&\simeq &\pi_{2i}\left(HH(A/R,\bz_p)\widehat{\otimes}_R \pi_{2j} THH(S,\bz_p)[2j]\right)^{t\bt}[2i]\\
&\simeq &(\pi_{2(i-j)}\left(HH(A/R,\bz_p)\widehat{\otimes}_R \pi_{2j} THH(S,\bz_p)\right)^{t\bt})[2i].
\end{eqnarray*}
The Tate spectral sequence computing 
$$\mathrm{gr}_{\mathcal{Z}_S}^j  TP(A\widehat{\otimes}_R S,\bz_p)\simeq \left(HH(A/R,\bz_p)\widehat{\otimes}_R \pi_{2j} THH(S,\bz_p) [2j]\right)^{t\bt}$$ 
degenerates (as $E_2^{a,b}=0$ if $a$ or $b$ is odd by Lemma \ref{lemflat}), and
gives a decreasing complete exhaustive abutment filtration on 
$$\pi_{2i}\left(\mathrm{gr}_{\mathcal{Z}_S}^j  TP(A\widehat{\otimes}_R S,\bz_p)\right)\simeq \pi_{2(i-j)}\left(HH(A/R,\bz_p)\widehat{\otimes}_R \pi_{2j} THH(S,\bz_p)\right)^{t\bt}$$ 
with (after re-indexing) $k$-graded piece
\begin{eqnarray*}
\pi_{2k} \left(HH(A/R,\bz_p)\widehat{\otimes}_R \pi_{2j} THH(S,\bz_p)\right)&\simeq &\pi_{2k} HH(A/R,\bz_p)\widehat{\otimes}_R \pi_{2j} THH(S,\bz_p)\\
&\simeq& (L\Lambda^kL_{A/R})^{\wedge}_p[-k] \widehat{\otimes}_R \pi_{2j} THH(S,\bz_p).
\end{eqnarray*}
Similarly, the Tate spectral sequence computing 
$HP(A,\bz_p)$ 
degenerates and
gives a decreasing complete exhaustive abutment filtration on $\pi_{2(i-j)}HP(A/R,\bz_p)$. Applying $(-)\widehat{\otimes}_R \pi_{2j} THH(S,\bz_p)$, we obtain a decreasing exhaustive abutment filtration on 
$\pi_{2(i-j)}HP(A/R,\bz_p)\widehat{\otimes}_R \pi_{2j} THH(S,\bz_p)$
with $k$-graded piece
$$\pi_{2k} HH(A/R,\bz_p)\widehat{\otimes}_R \pi_{2j} THH(S,\bz_p)\simeq (L\Lambda^kL_{A/R})^{\wedge}_p[-k] \widehat{\otimes}_R \pi_{2j} THH(S,\bz_p).$$
The morphism induced by (\ref{thatmap}) on $i$-graded pieces
\begin{eqnarray}
\label{noncompletegrpHodge}&&(L\widehat{\Omega}_{A/R})^{\wedge}_p\widehat{\otimes}_R \pi_{2j} THH(S,\bz_p)\\
\label{noncompletegrp}&\simeq&\left(\pi_{ 2(i-j)}HP(A/R,\bz_p)\right)\widehat{\otimes}_R \pi_{2j} THH(S,\bz_p)\\
\label{completegrp}&\longrightarrow& \pi_{2(i-j)}\left(HH(A/R,\bz_p)\widehat{\otimes}_R \pi_{2j} THH(S,\bz_p)\right)^{t\bt}
\end{eqnarray}
is compatible with the abutment filtrations mentioned above, and induces equivalences on graded pieces.
However, the abutment filtration on (\ref{noncompletegrp}), which corresponds to the filtration $(L\widehat{\Omega}^{\geq *}_{A/R})^{\wedge}_p\widehat{\otimes}_R \pi_{2j} THH(S,\bz_p)$ on (\ref{noncompletegrpHodge}), might not be complete. By contrast, the abutment filtration obtained on the target of (\ref{completegrp}) is complete. In other words, the map (\ref{completegrp}) is the completion map of the source with respect to the filtration induced by the Tate spectral sequence. See Remark \ref{rem-marc} for more details.

Hence the map of filtrations (\ref{thatmap}) induces on $i$-graded pieces the canonical map from
$$(L\widehat{\Omega}_{A/R})^{\wedge}_p[2(i-j)]\widehat{\otimes}_R \pi_{2j} THH(S,\bz_p)[2j]$$
to
$$R\mathrm{lim}_k \left((L\Omega^{<k}_{A/R})^{\wedge}_p[2(i-j)]\widehat{\otimes}_R \pi_{2j} THH(S,\bz_p)[2j]\right)\simeq \mathrm{gr}_{\mathcal{Z}_S}^j\mathrm{gr}^{i}_{BMS}TP(A\widehat{\otimes}_R S,\bz_p)$$
where the limit is taken with respect to $k\geq 0$.\\

\end{proof}

\begin{lem}\label{lemflat}
Let $R,A,S$ as in Proposition \ref{prop1}. Let $M\in D(A)$ and $N\in D(S)$ be $p$-completely flat over $A$ and $S$ respectively. Then $M\widehat{\otimes}_R N \in D(A)$ is $p$-completely flat over $A$. In particular  $M\widehat{\otimes}_R N$ is discrete.
\end{lem}

\begin{proof}
Note that $N\in D(R)$ is $p$-completely flat over $R$ since $R\rightarrow S$ is $p$-completely flat and $N\in D(S)$ is $p$-completely flat over $S$, see Lemma \ref{trivlemma}. We need to show that $M\widehat{\otimes}_R N\otimes^L_{A}A/pA$ is concentrated in degree $0$ and given by a flat  $A/pA$-module. We have
\begin{eqnarray*}
M\widehat{\otimes}_R N\otimes^L_{A}A/pA&=& (M\otimes^L_R N)^{\wedge}_p\otimes^L_{A}A/pA \\
&\simeq& N\otimes^L_R M\otimes^L_{A}A/pA\\
&\simeq& N\otimes^L_R M/pM\\
&\simeq& N\otimes^L_RR/pR \otimes^L_{(R/pR)} M/pM\\
&\simeq& N/pN \otimes^L_{(R/pR)} M/pM\\
&\simeq& N/pN \otimes_{(R/pR)} M/pM
\end{eqnarray*}
which is indeed concentrated in degree $0$. Moreover $N/pN \otimes_{(R/pR)} M/pM$ is a flat $(A/pA)$-module, as $M/pM$ is a flat $(A/pA)$-module and $N/pN$ is a flat $(R/pR)$-module.  
\end{proof}

\begin{rem}\label{rem-marc}
Let $R,A,S$ as in Proposition \ref{prop1}.  We set $N:=\pi_{2j}THH(S,\bz_p)$. Recall that $N$ is derived $p$-complete and $p$-completely flat over $R$. Let $P_{\bullet}\rightarrow A$ be a simplicial resolution of $A$ by flat $R$-algebras. We consider the cyclic simplicial $R$-module 
$$\appl{\Lambda^{\mathrm{op}}}{\mathrm{Fun}(\Delta^{\mathrm{op}},\mathrm{Mod}_R)}{\mathrm{[}n\mathrm{]}}{P_{\bullet}^{\otimes_{R}{n+1}}}$$
which we see as a cyclic object in $\mathrm{Ch}_{\geq 0}(R)$, using the unnormalized chain complex functor. Here $\Lambda$ denotes the cyclic category and $\mathrm{Mod}_R$ denotes the category of usual (i.e. discrete) $R$-modules. Let $(C_*(A/R),b)$ be the total complex of the corresponding double chain complex, so that $(C_*(A/R),b)\simeq HH(A/R)$ which is concentrated in even homological degrees, and let $\mathrm{BP}(A/R)$ be the associated periodic bicomplex, see \cite[Section 2]{Hoyois15}. Let $\mathrm{Tot}(\mathrm{BP}(A/R))$ be the associated total complex with respect to the direct product. We have $\mathrm{Tot}(\mathrm{BP}(A/R))\simeq HP(A/R)$ by \cite[Theorem 2.1]{Hoyois15}. Hence we have 
$$\mathrm{Tot}(\mathrm{BP}(A/R)^{\wedge}_p) \simeq \mathrm{Tot}(\mathrm{BP}(A/R))^{\wedge}_p\simeq HP(A/R,\bz_p).$$
Then $\mathrm{Tot}(\mathrm{BP}(A/R)^{\wedge}_p)$ has a decreasing complete exhaustive filtration $F^*$ with graded pieces 
\begin{equation}\label{gr1}
\mathrm{gr}^n_F\mathrm{Tot}(\mathrm{BP}(A/R)^{\wedge}_p)\simeq HH(A/R,\bz_p)[-2n].
\end{equation}
Hence the complex $\mathrm{Tot}(\mathrm{BP}(A/R)^{\wedge}_p)\widehat{\otimes}_RN$ has a decreasing, possibly non-complete, exhaustive filtration with graded pieces
$$\left(\mathrm{gr}^n_F\mathrm{Tot}(\mathrm{BP}(A/R)^{\wedge}_p)\right) \widehat{\otimes}_RN\simeq HH(A/R,\bz_p)[-2n] \widehat{\otimes}_RN.$$
Similarly, we consider the cyclic simplicial $R$-module 
$$\appl{\Lambda^{\mathrm{op}}}{\mathrm{Fun}(\Delta^{\mathrm{op}},\mathrm{Mod}_R)}{\mathrm{[}n\mathrm{]}}{(P_{\bullet}^{\otimes_{R}{n+1}})^{\wedge}_p\widehat{\otimes}_RN}$$
with associated total chain complex $(C_*(A/R)^{\wedge}_p,b)\widehat{\otimes}_RN$. Note that each $(P_{i}^{\otimes_{R}{n+1}})^{\wedge}_p\widehat{\otimes}_RN$ is $p$-completely flat over $R$, hence concentrated in degree $0$. 
We consider the corresponding periodic bicomplex $\mathrm{BP}(A/R)^{\wedge}_p\widehat{\otimes}_RN$ and its total complex 
$$\mathrm{Tot}(\mathrm{BP}(A/R)^{\wedge}_p\widehat{\otimes}_RN)\simeq (HH(A/R,\bz_p)\widehat{\otimes}_RN)^{t\bt}$$ 
with respect to the direct product.  Then $\mathrm{Tot}(\mathrm{BP}(A/R)^{\wedge}_p\widehat{\otimes}_RN)$ has a decreasing complete exhaustive filtration $F^*$ with graded pieces 
\begin{equation}\label{gr2}
\mathrm{gr}^n_F\mathrm{Tot}(\mathrm{BP}(A/R)^{\wedge}_p\widehat{\otimes}_RN)\simeq HH(A/R,\bz_p)[-2n]\widehat{\otimes}_RN.
\end{equation}
There is a canonical map of exhaustive filtrations with complete target
$$\left(F^*\mathrm{Tot}(\mathrm{BP}(A/R)^{\wedge}_p)\right) \widehat{\otimes}_RN\longrightarrow F^*\mathrm{Tot}\left(\mathrm{BP}(A/R)^{\wedge}_p \widehat{\otimes}_RN\right)$$
which induces equivalences on graded pieces. Moreover, the complexes (\ref{gr1}) and (\ref{gr2}) are concentrated in even homological degrees for any $n$, hence the associated spectral sequence degenerates.
Therefore, for any fixed $i\in\bz$,  
we obtain a map of decreasing exhaustive filtrations with complete target
$$\left(\pi_{2i}F^*\mathrm{Tot}(\mathrm{BP}(A/R)^{\wedge}_p)\right) \widehat{\otimes}_RN\longrightarrow \pi_{2i}F^*\mathrm{Tot}\left(\mathrm{BP}(A/R)^{\wedge}_p \widehat{\otimes}_RN\right)$$
inducing equivalences on graded pieces, so that the filtration on the right hand side may be identified with the completion of the filtration on the left hand side.
\end{rem}

\begin{condition}\label{condlimtens}
Let $R\in \mathrm{QSyn}_{\bz_p}$ be a quasisyntomic ring. Consider the following condition: For any $j\geq 0$, the functor
$$
\appl{D(R)^{\mathrm{Cpl}(p)}}{D(R)^{\mathrm{Cpl}(p)}}{K}{K\widehat{\otimes}_R\,\mathrm{gr}^j_{BMS}THH(R,\bz_p)}
$$
commutes with small limits.
\end{condition} 
\begin{example}\label{exampleperfect}
Any perfectoid $R$ satisfies Condition \ref{condlimtens} by \cite[Theorem 6.1]{BMS}, and so does $\bz_p$ by Corollary \ref{key}.  
\end{example}
\begin{rem}\label{rem-dualizable}
Let $R\in \mathrm{QSyn}_{\bz_p}$ be a quasisyntomic ring such that $\mathrm{gr}^j_{BMS}THH(R,\bz_p)\in D(R)$ is perfect \cite[Definition 7.2.4.1]{LurieHA} for any $j\geq 0$. It follows from \cite[Proposition 7.2.4.4, Proposition 1.4.4.4]{LurieHA} that $R$ satisfies  Condition \ref{condlimtens}.
\end{rem}

\begin{prop}\label{filtQR}
Let $R\in \mathrm{QSyn}_{\bz_p}$ satisfying Condition \ref{condlimtens}, and let $A\in \mathrm{qrsPerfd}_{R}$. If $R$ is perfectoid or $R=\bz_p$, we allow more generally $A\in \mathrm{QRSPerfd}_{R}$. 
\begin{enumerate}
\item There is an $\bn^{\mathrm{op}}$-indexed decreasing multiplicative complete exhaustive $\bt$-equivariant filtration 
$\mathcal{Z}_R^*\mathrm{Fil}^*_{BMS} THH(A,\bz_p)$
on $\mathrm{Fil}^*_{BMS}THH(A,\bz_p)$
 endowed with equivalences
$$\mathrm{gr}^j_{\mathcal{Z}_R} \mathrm{Fil}^{*}_{BMS}THH(A,\bz_p)\simeq \left( \mathrm{Fil}^{*-j}_{BMS}HH(A/R,\bz_p)\right)\widehat{\otimes}_{R}\mathrm{gr}^{j}_{BMS}THH(R,\bz_p).$$

\item  For $?=P,C^-$, there is an  $\bn^{\mathrm{op}}$-indexed decreasing multiplicative complete exhaustive filtration 
$\mathcal{Z}_R^*\mathrm{Fil}^*_{BMS} T?(A,\bz_p)$ 
on $\mathrm{Fil}^*_{BMS}T?(A,\bz_p)$ 
 endowed with equivalences
$$\mathrm{gr}^j_{\mathcal{Z}_R} \mathrm{Fil}^{*}_{BMS}T?(A,\bz_p)\simeq \left( \mathrm{Fil}^{*-j}_{BMS}H?(A/R,\bz_p)\right)\widehat{\otimes}_{R}\mathrm{gr}^{j}_{BMS}THH(R,\bz_p).$$
\end{enumerate}
\end{prop}
\begin{proof}
(1) Let $R\rightarrow S$ be a quasisyntomic cover with $S\in \mathrm{QRSPerfd}_{R}$ and  let $S^{\bullet}$ be its Cech nerve. Then  $S^{i}\in \mathrm{QRSPerfd}_{R}$ and $R\rightarrow S^{i}$ is a quasisyntomic cover for all $i\geq 0$. Moreover
the cosimplicial ring $A\widehat{\otimes}_R S^{\bullet}$ is isomorphic to the Cech nerve $(A\widehat{\otimes}_RS)^{\bullet}$ of the quasisyntomic cover $A\rightarrow A\widehat{\otimes}_RS$. Hence we have
$$THH(A,\bz_p)\simeq \mathrm{lim}_{\Delta}(THH(A\widehat{\otimes}_R S^{\bullet} ,\bz_p))$$
and 
$$\mathrm{Fil}^*_{BMS} THH(A,\bz_p)\simeq \mathrm{lim}_{\Delta}(\mathrm{Fil}^*_{BMS}THH(A\widehat{\otimes}_R S^{\bullet},\bz_p)),$$
where the limit is computed with respect to the cosimplicial direction.
We consider the cosimplicial bifiltration $\mathcal{Z}_{S^{\bullet}}^*\mathrm{Fil}^*_{BMS}THH(A\widehat{\otimes}_R S^{\bullet},\bz_p)$.
We obtain a bicomplete biexhaustive bifiltration
$$\mathcal{Z}_{S/R}^*\mathrm{Fil}^*_{BMS} THH(A,\bz_p):=\mathrm{lim}_{\Delta}\left(\mathcal{Z}_{S^{\bullet}}^*\mathrm{Fil}^*_{BMS}THH(A\widehat{\otimes}_R S^{\bullet},\bz_p)\right)$$
with graded pieces 
\begin{eqnarray}
&&\mathrm{gr}^j_{\mathcal{Z}_{S/R}} \mathrm{gr}^{n}_{BMS}THH(A,\bz_p)\\
&\simeq &\mathrm{lim}_{\Delta}\left(\mathrm{gr}^j_{\mathcal{Z}_{S^{\bullet}}}\mathrm{gr}^n_{BMS}THH(A\widehat{\otimes}_R S^{\bullet},\bz_p)\right)\\
&\simeq &\mathrm{lim}_{\Delta}\left((L\Lambda^{n-j}L_{A/R})^{\wedge}_p\widehat{\otimes}_{R}\pi_{2j}THH(S^{\bullet},\bz_p)[2j]\right)[n-j]\\
\label{eqere}&\stackrel{\sim}{\leftarrow} &(L\Lambda^{n-j}L_{A/R})^{\wedge}_p\widehat{\otimes}_{R}\left(\mathrm{lim}_{\Delta}(\pi_{2j}THH(S^{\bullet},\bz_p)[2j])\right)[n-j]\\
&\simeq &(L\Lambda^{n-j}L_{A/R})^{\wedge}_p\widehat{\otimes}_{R}\mathrm{gr}^{j}_{BMS}THH(R,\bz_p)[n-j].
\end{eqnarray}
In order to check that (\ref{eqere}) is indeed an equivalence, we note that each term $$(L\Lambda^{n-j}L_{A/R})^{\wedge}_p[j-n]\widehat{\otimes}_{R}\pi_{2j}THH(S^{i},\bz_p)$$ is concentrated in degree $0$ by Lemma \ref{lemflat}.
We obtain a map of complete filtrations
\begin{eqnarray*}
\mathrm{gr}^j_{\mathcal{Z}_{S/R}} \mathrm{Fil}^{*}_{BMS}THH(A,\bz_p)&\simeq & \mathrm{lim}_{\Delta}\left(\mathrm{gr}^j_{\mathcal{Z}_{S^{\bullet}}} \mathrm{Fil}^{*}_{BMS}THH(A\widehat{\otimes}_R S^{\bullet},\bz_p)\right)\\
&\simeq &\mathrm{lim}_{\Delta}\left(\mathrm{Fil}^{*-j}_{BMS}HH(A/R,\bz_p)\widehat{\otimes}_{R}\pi_{2j}THH(S^{\bullet},\bz_p)[2j]\right)\\
&\stackrel{\sim}{\leftarrow} &\mathrm{Fil}^{*-j}_{BMS}HH(A/R,\bz_p)\widehat{\otimes}_{R}\left(\mathrm{lim}_{\Delta}(\pi_{2j}THH(S^{\bullet},\bz_p)[2j])\right)\\
&\simeq &\mathrm{Fil}^{*-j}_{BMS}HH(A/R,\bz_p)\widehat{\otimes}_{R}\mathrm{gr}^{j}_{BMS}THH(R,\bz_p).
\end{eqnarray*}
which is an equivalence as it induces equivalences on graded pieces. These two filtrations are indeed complete since a limit of complete filtrations is complete, and since $(-)\widehat{\otimes}_{R}\mathrm{gr}^{j}_{BMS}THH(R,\bz_p)$ commutes with arbitrary small limits by Corollary \ref{key} and Condition \ref{condlimtens}.

The bifiltration $\mathcal{Z}_{S/R}^*\mathrm{Fil}^*_{BMS} THH(A,\bz_p)$ is functorial in $S$. More precisely, if $S\rightarrow S'$ is a morphism of quasisyntomic covers of $R$ such that $S,S'\in\mathrm{QRSPerfd}_{R}$, then we have a morphism
of bicomplete bifiltrations 
\begin{equation*}\label{biiso}
\mathcal{Z}_{S/R}^*\mathrm{Fil}^*_{BMS} THH(A,\bz_p)\rightarrow \mathcal{Z}_{S'/R}^*\mathrm{Fil}^*_{BMS} THH(A,\bz_p)
\end{equation*}
which is an equivalence, as it induces equivalences on graded pieces. We define
$$\mathcal{Z}_{R}^*\mathrm{Fil}^*_{BMS} THH(A,\bz_p):=\mathrm{colim}_{S/R} \left(\mathcal{Z}_{S/R}^*\mathrm{Fil}^*_{BMS} THH(A,\bz_p) \right)$$
where the colimit is taken over the category of quasisyntomic covers $R\rightarrow S$ with $S\in\mathrm{QRSPerfd}_{R}$.

(2) We give the argument for $TP$; the proof for $TC^-$ is the same. Let $R\rightarrow S$ be a quasisyntomic cover with $S\in \mathrm{QRSPerfd}_{R}$. We define the bifiltration
$$\mathcal{Z}_{S/R}^*\mathrm{Fil}^*_{BMS} TP(A,\bz_p):=\mathrm{lim}_{\Delta}\left(\mathcal{Z}_{S^{\bullet}}^*\mathrm{Fil}^*_{BMS}TP(A\widehat{\otimes}_R S^{\bullet},\bz_p)\right)$$
with graded pieces 
\begin{eqnarray}
&&\mathrm{gr}^j_{\mathcal{Z}_{S/R}} \mathrm{gr}^{n}_{BMS}TP(A,\bz_p)\\
&\simeq &\mathrm{lim}_{\Delta}\, \mathrm{lim}_k\left((L\Omega^{<k}_{A/R})^{\wedge}_p[2(n-j)]\widehat{\otimes}_{R}\pi_{2j}THH(S^{\bullet},\bz_p)[2j]\right)\\
&\simeq & \mathrm{lim}_k\, \mathrm{lim}_{\Delta}\left((L\Omega^{<k}_{A/R})^{\wedge}_p[2(n-j)]\widehat{\otimes}_{R}\pi_{2j}THH(S^{\bullet},\bz_p)[2j]\right)\\
\label{eqere2}&\stackrel{\sim}{\leftarrow}& \mathrm{lim}_k\left((L\Omega^{<k}_{A/R})^{\wedge}_p[2(n-j)]\widehat{\otimes}_{R} \, \mathrm{lim}_{\Delta}(\pi_{2j}THH(S^{\bullet},\bz_p)[2j])\right)\\
&\simeq & \mathrm{lim}_k\left((L\Omega^{<k}_{A/R})^{\wedge}_p[2(n-j)]\widehat{\otimes}_{R} \, \mathrm{gr}^{j}_{BMS}THH(R,\bz_p)\right)\\
\label{eqere3}&\stackrel{\sim}{\leftarrow} &(L\widehat{\Omega}_{A/R})^{\wedge}_p\widehat{\otimes}_{R} \mathrm{gr}^{j}_{BMS}THH(R,\bz_p)[2(n-j)].
\end{eqnarray}
One checks that (\ref{eqere2}) is an equivalence using the Hodge filtration and the fact that (\ref{eqere}) is an equivalence. Moreover,  (\ref{eqere3}) is an equivalence since  $(-)\widehat{\otimes}_{R}\mathrm{gr}^{j}_{BMS}THH(R,\bz_p)$ commutes with arbitrary limits by Corollary \ref{key} and Condition \ref{condlimtens}.

We obtain a morphism of complete filtrations
\begin{eqnarray*}
&&\mathrm{gr}^j_{\mathcal{Z}_{S/R}} \mathrm{Fil}^{*}_{BMS}TP(A,\bz_p)\\
&\simeq &\mathrm{lim}_{\Delta} \left(\mathrm{gr}^j_{\mathcal{Z}_{S^{\bullet}}} \mathrm{Fil}^{*}_{BMS}TP(A\widehat{\otimes}_R S^{\bullet},\bz_p)\right)\\
&\leftarrow &\mathrm{lim}_{\Delta}\left(\mathrm{Fil}^{*-j}_{BMS}HP(A/R,\bz_p)\widehat{\otimes}_{R}\pi_{2j}THH(S^{\bullet},\bz_p)[2j]\right)\\
&\leftarrow &\left(\mathrm{Fil}^{*-j}_{BMS}HP(A/R,\bz_p) \right)\widehat{\otimes}_{R} \mathrm{lim}_{\Delta}\left(\pi_{2j}THH(S^{\bullet},\bz_p)[2j]\right)\\
&\simeq &\left(\mathrm{Fil}^{*-j}_{BMS}HP(A/R,\bz_p) \right)\widehat{\otimes}_{R} \mathrm{gr}^{j}_{BMS}THH(R,\bz_p).
\end{eqnarray*}
which is an equivalence since it induces equivalences on graded pieces. We define
$$\mathcal{Z}_{R}^*\mathrm{Fil}^*_{BMS} TP(A,\bz_p):=\mathrm{colim}_{S/R} \left(\mathcal{Z}_{S/R}^*\mathrm{Fil}^*_{BMS} TP(A,\bz_p) \right)$$
where the colimit is taken over the category of quasisyntomic covers $R\rightarrow S$ with $S\in\mathrm{QRSPerfd}_{R}$. 

\end{proof}

\begin{fff-proof} Assume that $R$ is perfectoid or $R=\bz_p$, and $A\in \mathrm{QSyn}_{R}$. The filtration $\mathrm{Fil}^*_{BMS} THH(-,\bz_p)$ (resp. $\mathrm{Fil}^*_{BMS} T?(-,\bz_p)$) is  a $\widehat{DF}(\bs[\bt])$-valued (resp. $\widehat{DF}(\bs)$-valued) sheaf  on $\mathrm{QRSPerfd}^{\mathrm{op}}_{R}$. For $?=HH,P,C^-$, 
$\mathcal{Z}_{R}^*\mathrm{Fil}^*_{BMS} T?(-,\bz_p)$ is an $\bn^{\mathrm{op}}$-indexed multiplicative complete exhaustive filtration on the sheaf $\mathrm{Fil}^*_{BMS} T?(-,\bz_p)$ on $\mathrm{QRSPerfd}^{\mathrm{op}}_{R}$, by Proposition \ref{filtQR}. Hence the filtration $$\mathcal{Z}_{R}^*\mathrm{Fil}^*_{BMS} T?(A,\bz_p):=R\Gamma_{\mathrm{syn}}(A,\mathcal{Z}_{R}^*\mathrm{Fil}^*_{BMS} T?(-,\bz_p))$$ is also $\bn^{\mathrm{op}}$-indexed, multiplicative, complete and exhaustive (and $\bt$-equivariant for $?=HH$). For $?=P,C^-$, we have
\begin{eqnarray*}
&&\mathrm{gr}^j_{\mathcal{Z}_{R}} \mathrm{Fil}^{*}_{BMS}T?(A,\bz_p)\\
&\simeq& R\Gamma_{\mathrm{syn}}\left(A,\mathrm{gr}^j_{\mathcal{Z}_{R}}\mathrm{Fil}^*_{BMS} T?(-,\bz_p)\right)\\
&\simeq& R\Gamma_{\mathrm{syn}}\left(A, \mathrm{Fil}^{*-j}_{BMS}H?(-/R,\bz_p)\widehat{\otimes}_{R} \mathrm{gr}^{j}_{BMS}THH(R,\bz_p)\right)\\
&\simeq& R\Gamma_{\mathrm{syn}}\left(A, \mathrm{Fil}^{*-j}_{BMS}H?(-/R,\bz_p)\right)\widehat{\otimes}_{R} \mathrm{gr}^{j}_{BMS}THH(R,\bz_p)\\
&\simeq& \left(\mathrm{Fil}^{*-j}_{BMS}H?(A/R,\bz_p)\right)\widehat{\otimes}_{R} \mathrm{gr}^{j}_{BMS}THH(R,\bz_p)
\end{eqnarray*}
using Proposition \ref{filtQR} and the fact that $(-)\widehat{\otimes}_{R} \mathrm{gr}^{j}_{BMS}THH(R,\bz_p)$ commutes with small limits by Example \ref{exampleperfect}. The second equivalence of the theorem follows for $R$ perfectoid or $R=\bz_p$; the proof of the first equivalence is the same.

Let $R\in \mathrm{QSyn}_{\bz_p}$ such that $(L\Lambda^{j}L_{R/\bz_p})^{\wedge}_p$ is a perfect complex of $R$-modules for all $j\geq0$. By the previous paragraph, $\mathrm{gr}^{n}_{BMS}THH(R,\bz_p)$ has a finite complete exhaustive filtration $\mathcal{Z}_{\bz_p}^*\mathrm{gr}^n_{BMS} THH(R,\bz_p)$ with $j$-graded piece given by (\ref{kill}), see Corollary \ref{corzp}. Moreover, we have $\mathcal{Z}_{\bz_p}^0\mathrm{gr}^0_{BMS} THH(R,\bz_p)\simeq R$, so that $\mathcal{Z}_{\bz_p}^*\mathrm{gr}^n_{BMS} THH(R,\bz_p)\in \widehat{DF}(R)$ by multiplicativity of the bifiltration $\mathcal{Z}_{\bz_p}^*\mathrm{Fil}^*_{BMS} THH(R,\bz_p)$. Hence $\mathrm{gr}^n_{BMS} THH(R,\bz_p)$ is a perfect complex in $D(R)$, for any $n\geq 0$. Thus by Remark \ref{rem-dualizable}, Proposition \ref{filtQR} applies. Working in the small site $\mathrm{qrsPerfd}^{\mathrm{op}}_{R}$, we obtain the result for any $A\in \mathrm{qSyn}_{R}$ as in the previous paragraph.

\end{fff-proof}

\begin{rem}\label{remforbifilt}
Going through the construction of $\mathcal{Z}_{R}^*$, one checks that the map 
$$\mathcal{Z}_{R}^0\mathrm{Fil}^*_{BMS} THH(A,\bz_p)\rightarrow \mathrm{gr}^0_{\mathcal{Z}_{R}} \mathrm{Fil}^{*}_{BMS}THH(A,\bz_p)$$
is equivalent to the canonical map (see Proposition \ref{mapBMS-T-H})
$$\mathrm{Fil}^*_{BMS} THH(A,\bz_p)\rightarrow \mathrm{Fil}^*_{BMS} HH(A/R,\bz_p)$$
and similarly for $TP(A,\bz_p)$ and $TC^-(A,\bz_p)$. 
\end{rem}

\begin{cor}\label{corzp}
For any  $A\in\mathrm{QSyn}_{\bz_p}$ and any $n\in\bz$, we have $\bn^{op}$-indexed complete exhaustive filtrations on $\mathrm{gr}_{BMS}^n$, endowed canonical equivalences
\begin{eqnarray}
\label{kill}\mathrm{gr}^j_{\mathcal{Z}_{\bz_p}} \mathrm{gr}^{n}_{BMS}THH(A,\bz_p)&\simeq & L\Lambda^{n-j}L_{A/\bz_p}\widehat{\otimes}_{\bz_p}\bz_p/j\bz_p[n+j-\epsilon_j]\\
\mathrm{gr}^j_{\mathcal{Z}_{\bz_p}} \mathrm{gr}^{n}_{BMS}TP(A,\bz_p)&\simeq & L\widehat{\Omega}_{A/\bz_p}\widehat{\otimes}_{\bz_p}\bz_p/j\bz_p[2n-\epsilon_j]\\
\mathrm{gr}^j_{\mathcal{Z}_{\bz_p}} \mathrm{gr}^{n}_{BMS}TC^-(A,\bz_p)&\simeq & L\widehat{\Omega}_{A/\bz_p}^{\geq n-j}\widehat{\otimes}_{\bz_p}\bz_p/j\bz_p[2n-\epsilon_j]
\end{eqnarray}
for any $j\geq 0$ and any $n\in\bz$, where $\epsilon_j:=\mathrm{Min}(1,j)$.
\end{cor}
\begin{proof}
This follows from Theorem \ref{thmpadic}, Corollary \ref{key} and  \cite[Theorem 1.17]{BMS}, see Section \ref{sect-BMS}.
\end{proof}

There is some overlap between (\ref{kill}) and the results of \cite[Sections 5.2, 5.3]{Antieau-Mathew-Morrow-Nikolaus20}. For instance \cite[Corollary 5.21, 5.22]{Antieau-Mathew-Morrow-Nikolaus20} follow from (\ref{kill}). In particular, we have the following
\begin{cor}
For any $A\in \mathrm{QSyn}_{\bz_p}$, we have $\mathrm{gr}^{n}_{BMS}THH(A,\bz_p)\in \mathrm{Sp}_{\geq n}$ and $\mathrm{Fil}^{n}_{BMS}THH(A,\bz_p) \in \mathrm{Sp}_{\geq n}$. The functors $A\mapsto \mathrm{gr}^{n}_{BMS}THH(A,\bz_p)$ and 
$A\mapsto \mathrm{Fil}^{n}_{BMS}THH(A,\bz_p)$ are left Kan extended from finitely generated $p$-complete polynomial $\bz_p$-algebras, as functors to $p$-complete spectra.
\end{cor}

\subsection{The bifiltration for global rings}\label{SectionBifiltglobal}

Let $A$ be a quasi-lci ring with bounded torsion. We define (trivial) multiplicative complete exhaustive $(\bn^{\mathrm{op}}\times\bz^{\mathrm{op}})$-indexed bifiltrations on Hochschild homology and its variants as follows. We set
$$\mathcal{Z}_{\bz}^0\mathrm{Fil}_{HKR}^*HH(A):=\mathrm{Fil}_{HKR}^*HH(A) \textrm{ and } \mathcal{Z}_{\bz}^j\mathrm{Fil}_{HKR}^*HH(A):=0 \textrm{ for any } j\geq1,$$
and for $?=P,C^-$ we set 
$$\mathcal{Z}_{\bz}^0\mathrm{Fil}_{B}^*H?(A):=\mathrm{Fil}_{B}^*H?(A) \textrm{ and } \mathcal{Z}_{\bz}^j\mathrm{Fil}_{B}^*H?(A):=0 \textrm{ for any } j\geq1.$$
Similarly, for any prime $p$ and $?=H,P,C^-$, we set
$$\mathcal{Z}_{\bz_p}^0\mathrm{Fil}_{BMS}^*H?(A^{\wedge}_p/\bz_p,\bz_p):=\mathrm{Fil}_{BMS}^*H?(A^{\wedge}_p/\bz_p,\bz_p)$$ 
$$\textrm{ and } \mathcal{Z}_{\bz_p}^j\mathrm{Fil}_{BMS}^*H?(A^{\wedge}_p/\bz_p,\bz_p):=0 \textrm{ for any } j\geq1.$$

\begin{defn}\label{global-bifilt}
Let $A$ be a quasi-lci ring with bounded torsion. We define multiplicative bicomplete $(\bn^{\mathrm{op}}\times\bz^{\mathrm{op}})$-indexed bifiltrations
\[ \xymatrix{
\mathcal{Z}_{\bz}^*F^*THH(A)\ar[d]\ar[r]& \mathcal{Z}_{\bz}^*\mathrm{Fil}_{HKR}^*HH(A) \ar[d]\\
\prod_p\mathcal{Z}_{\bz_p}^*\mathrm{Fil}_{BMS}^*THH(A^{\wedge}_p,\bz_p)\ar[r]&\prod_p\mathcal{Z}_{\bz_p}^*\mathrm{Fil}_{BMS}^*HH(A^{\wedge}_p/\bz_p,\bz_p) 
}
\]
and, for $?=P,C^-$
\[ \xymatrix{
\mathcal{Z}_{\bz}^*F^*T?(A)\ar[d]\ar[r]& \mathcal{Z}_{\bz}^*\mathrm{Fil}_{B}^*H?(A) \ar[d]\\
\prod_p\mathcal{Z}_{\bz_p}^*\mathrm{Fil}_{BMS}^*T?(A^{\wedge}_p,\bz_p)\ar[r]&\prod_p\mathcal{Z}_{\bz_p}^*\mathrm{Fil}_{BMS}^*H?(A^{\wedge}_p/\bz_p,\bz_p) 
}
\]
as fiber products of $\mathbb{E}_{\infty}$-algebra objects in the symmetric monoidal $\infty$-categories $\widehat{DBF}(\bs[\bt])$ and $\widehat{DBF}(\bs)$ respectively.

There is an evident morphism $$\mathcal{Z}_{\bz}^*F^*TC^-(A)\rightarrow \mathcal{Z}_{\bz}^*F^*TP(A)$$ of $\mathbb{E}_{\infty}$-algebra objects, and we define $$\mathcal{Z}_{\bz}^*F^*\Sigma^2TC^+(A)\in \widehat{DBF}(\bs)$$ as its cofiber computed in 
$\widehat{DBF}(\bs)$.
\end{defn}

\begin{thm}\label{thmfinal} Let $A$ be a quasi-lci ring with bounded torsion. 
\begin{enumerate}
\item For 
$$?= THH(A), TP(A) ,TC^-(A),$$
the filtration $\mathcal{Z}_{\bz}^*F^*?$ on $F^*?$ is functorial, $\bn^{\mathrm{op}}$-indexed, multiplicative, complete and exhaustive. 

\item We have canonical equivalences
\begin{eqnarray*}
\mathrm{gr}^j_{\mathcal{Z}_{\bz}} F^*THH(A)&\simeq &\left( \mathrm{Fil}_{HKR}^{*-j}HH(A)\right)\otimes^L_{\bz}\mathrm{gr}^{j}_{F}THH(\bz);\\
\mathrm{gr}^j_{\mathcal{Z}_{\bz}} F^*TP(A) &\simeq & \left( \mathrm{Fil}_B^{*-j}HP(A)\right)\otimes^L_{\bz}\mathrm{gr}^{j}_{F}THH(\bz);\\
\mathrm{gr}^j_{\mathcal{Z}_{\bz}} F^*TC^-(A)&\simeq & \left( \mathrm{Fil}_B^{*-j}HC^-(A)\right)\otimes^L_{\bz}\mathrm{gr}^{j}_{F}THH(\bz).
\end{eqnarray*}
\item In particular, there are complete exhaustive $\bn^{\mathrm{op}}$-indexed filtrations  $\mathcal{Z}^*_{\bz} \mathrm{gr}^{n}_{F}THH(A)$, $\mathcal{Z}^*_{\bz} \mathrm{gr}^{n}_{F}TP(A)$ and $\mathcal{Z}^*_{\bz} \mathrm{gr}^{n}_{F}TC^-(A)$ on 
the $H\bz$-modules $\mathrm{gr}^{n}_{F}THH(A)$, $\mathrm{gr}^{n}_{F}TP(A)$ and $\mathrm{gr}^{n}_{F}TC^-(A)$ respectively, endowed with canonical equivalences
\begin{eqnarray*}
\mathrm{gr}^j_{\mathcal{Z}_{\bz}} \mathrm{gr}^{n}_{F}THH(A)&\simeq &L\Lambda^{n-j}L_{A/\bz}\otimes^L_{\bz}\bz/j\bz[n+j-\epsilon_j];\\
\mathrm{gr}^j_{\mathcal{Z}_{\bz}} \mathrm{gr}^{n}_{F}TP(A) &\simeq & L\widehat{\Omega}_{A/\bz}\otimes^L_{\bz}\bz/j\bz[2n-\epsilon_j];\\
\mathrm{gr}^j_{\mathcal{Z}_{\bz}} \mathrm{gr}^{n}_{F}TC^-(A)&\simeq & L\widehat{\Omega}^{\geq n-j}_{A/\bz}\otimes^L_{\bz} \bz/j\bz[2n-\epsilon_j];\\
\mathrm{gr}^j_{\mathcal{Z}_{\bz}} \mathrm{gr}^{n}_{F} \Sigma^2TC^+(A)&\simeq& L\Omega^{<n-j}_{A/\bz}\otimes^L_{\bz}\bz/j\bz[2n-\epsilon_j].
\end{eqnarray*}
for any $j\in\bn$ and any $n\in\bz$, where $\epsilon_j:=\mathrm{Min}(1,j)$.
\end{enumerate}
\end{thm}
\begin{proof}
Assertion (1) is true by definition, except for the exhaustiveness, which is easy to check.

(2) For any $j\geq 0$, we have a cartesian square in $\widehat{DF}(\bs[\bt])$
\[ \xymatrix{
\mathrm{gr}^j_{\mathcal{Z}_{\bz}}F^*THH(A)\ar[d]\ar[r]& \mathrm{gr}^j_{\mathcal{Z}_{\bz}}\mathrm{Fil}_{HKR}^*HH(A) \ar[d]\\
\prod_p\mathrm{gr}^j_{\mathcal{Z}_{\bz_p}}\mathrm{Fil}_{BMS}^*THH(A^{\wedge}_p,\bz_p)\ar[r]&\prod_p\mathrm{gr}^j_{\mathcal{Z}_{\bz_p}}\mathrm{Fil}_{BMS}^*HH(A^{\wedge}_p/\bz_p,\bz_p) 
}
\]
For $j=0$, this gives the cartesian square
\[ \xymatrix{
\mathrm{gr}^0_{\mathcal{Z}_{\bz}}F^*THH(A)\ar[d]\ar[r]& \mathrm{Fil}_{HKR}^*HH(A) \ar[d]\\
\prod_p\mathrm{gr}^0_{\mathcal{Z}_{\bz_p}}\mathrm{Fil}_{BMS}^*THH(A^{\wedge}_p,\bz_p)\ar[r]^{\simeq}&\prod_p\mathrm{Fil}_{BMS}^*HH(A^{\wedge}_p/\bz_p,\bz_p) 
}
\]
where the lower horizontal map is an equivalence, see Remark \ref{remforbifilt}. Hence the upper horizontal map is an equivalence. For $j\geq 1$, we get the cartesian square
\[ \xymatrix{
\mathrm{gr}^j_{\mathcal{Z}_{\bz}}F^*THH(A)\ar[d]\ar[r]& 0 \ar[d]\\
\prod_p\mathrm{gr}^j_{\mathcal{Z}_{\bz_p}}\mathrm{Fil}_{BMS}^*THH(A^{\wedge}_p,\bz_p)\ar[r]&0 
}
\]
hence equivalences
\begin{eqnarray*}
\mathrm{gr}^j_{\mathcal{Z}_{\bz}}F^*THH(A)&\stackrel{\sim}{\rightarrow}&\prod_p\mathrm{gr}^j_{\mathcal{Z}_{\bz_p}}\mathrm{Fil}_{BMS}^*THH(A^{\wedge}_p,\bz_p)\\
\label{la2}&\simeq& \prod_p \mathrm{Fil}_{BMS}^{*-j}HH(A^{\wedge}_p/\bz_p,\bz_p)\widehat{\otimes}_{\bz_p}\mathrm{gr}_{BMS}^jTHH(\bz_p,\bz_p)\\
\label{la3}&\simeq & \prod_p \left(\mathrm{Fil}_{HKR}^{*-j}HH(A)\right)^{\wedge}_p\widehat{\otimes}_{\bz_p}\mathrm{gr}_{BMS}^jTHH(\bz_p,\bz_p)\\
\label{la4}&\simeq & \prod_p \left(\mathrm{Fil}_{HKR}^{*-j}HH(A)\right)^{\wedge}_p\widehat{\otimes}_{\bz_p}\left(\mathrm{gr}_{F}^jTHH(\bz)\right)^{\wedge}_p\\
\label{la5}&\simeq & \prod_p \left(\mathrm{Fil}_{HKR}^{*-j}HH(A)\otimes^L_{\bz}\mathrm{gr}_{F}^jTHH(\bz)\right)^{\wedge}_p\\
\label{la6}&\simeq & \left(\mathrm{Fil}_{HKR}^{*-j}HH(A)\otimes^L_{\bz}\mathrm{gr}_{F}^jTHH(\bz)\right)^{\wedge}\\
\label{la7}&\simeq & \mathrm{Fil}_{HKR}^{*-j}HH(A)\otimes^L_{\bz}\mathrm{gr}_{F}^jTHH(\bz).
\end{eqnarray*}
Here we use respectively Theorem \ref{thmpadic}, Proposition \ref{mapfiltHKR}, Proposition \ref{pcomplTHH}, the fact that the $p$-completion functor is monoidal and the equivalence $\mathrm{gr}_{F}^jTHH(\bz)\simeq \bz/j\bz[2j-1]$ given by Corollary \ref{key}.
The proof of the analogous statement  for $TP$ and $TC^-$ is similar. 

Assertion (3) then follows from the computation of the graded pieces of the HKR filtration and the Beilinson filtration \cite{Antieau18}. Moreover, we have $\mathcal{Z}_{\bz}^0 \mathrm{gr}^{0}_{F}THH(A)\simeq A$, hence $\mathcal{Z}_{\bz}^* \mathrm{gr}^{n}_{F}THH(A)\in \widehat{DF}(A)$ by multiplicativity of the bifiltration. Similarly, $\mathcal{Z}_{\bz}^0 \mathrm{gr}^{0}_{F}TC^-(A)= \mathrm{gr}^{0}_{F}TC^-(A)$ is a $\mathbb{E}_{\infty}$-$\bz$-algebra (see the proof of Proposition \ref{propmotfiltTP}(4)), hence $\mathcal{Z}_{\bz}^* \mathrm{gr}^{n}_{F}TC^-(A)$, $\mathcal{Z}_{\bz}^* \mathrm{gr}^{n}_{F}TP(A)$ and therefore $\mathcal{Z}_{\bz}^* \mathrm{gr}^{n}_{F}\Sigma^2TC^+(A)$ are objects of $\widehat{DF}(\bz)$.

\end{proof}

\begin{cor}\label{connectivity}
For any quasi-lci ring with bounded torsion $A$ and any $n\geq 0$, we have  
$F^{n}\Sigma^2TC^+(A) \in \mathrm{Sp}_{> n}$ and $\mathcal{Z}_{\bz}^jF^{n}\Sigma^2TC^+(A)\in \mathrm{Sp}_{\geq n+j}$. The $\mathrm{Sp}$-valued functors $A\mapsto \mathrm{gr}^{n}_{F}\Sigma^2TC^+(A)$, $A\mapsto F^{n}\Sigma^2TC^+(A)$ and $A\mapsto \mathcal{Z}_{\bz}^jF^{n}\Sigma^2TC^+(A)$ are left Kan extended from finitely generated polynomial $\bz$-algebras, for any $n,j\geq 0$.
\end{cor}
\begin{proof}
We first show that $\mathrm{gr}^{n}_{F}\Sigma^2TC^+(A)\in \mathrm{Sp}_{> n}$ for any fixed $n$. By Theorem \ref{thmfinal}, it is enough to show that $L\Omega^{<n-j}_{A/\bz}\otimes^L_{\bz}\bz/j\bz[2n-\epsilon_j]\in \mathrm{Sp}_{> n}$ for any $0\leq j< n$. In view of the Hodge filtration, this follows from the fact that 
$$L\Lambda^k L_{A/\bz}\otimes^L_{\bz} \bz/j[2n-k-\epsilon_j]\in \mathrm{Sp}_{> n}$$
for any $0\leq j< n$ and any $0\leq k< n-j$, which in turn easily follows from $L\Lambda^k L_{A/\bz}\otimes^L_{\bz} \bz/j\in \mathrm{Sp}_{\geq 0}$. We obtain $F^{n}\Sigma^2TC^+(A) \in \mathrm{Sp}_{> n}$ by completeness of the filtration $F^{*}\Sigma^2TC^+(A)$. Similarly, $\mathrm{gr}^j_{\mathcal{Z}_{\bz}} \mathrm{gr}^{n}_{F} \Sigma^2TC^+(A)$ belongs to $\mathrm{Sp}_{\geq n+j}$ hence so does $\mathcal{Z}_{\bz}^jF^{n}\Sigma^2TC^+(A)$ by completeness.

Since $A\mapsto L\Omega^{<n-j}_{A/\bz}$ is left Kan extended from from finitely generated polynomial $\bz$-algebras, the same is true for $A\mapsto \mathrm{gr}^{n}_{F}\Sigma^2TC^+(A)$. Consider the left Kan extension 
$dF^{*}\Sigma^2TC^+(-)$ of the $DF(\bs)$-valued functor $A\mapsto F^{*}\Sigma^2TC^+(A)$. For any quasi-lci ring with bounded torsion $A$, there is a canonical map of filtrations $dF^{*}\Sigma^2TC^+(A)\rightarrow F^{*}\Sigma^2TC^+(A)$, which induces equivalences on graded pieces, as $A\mapsto \mathrm{gr}^{n}_{F}\Sigma^2TC^+(A)$ is left Kan extended for any $n\geq 0$. Hence it is enough to check that the filtration $dF^{*}\Sigma^2TC^+(A)$ is complete. But  $dF^{n}\Sigma^2TC^+(A)\in \mathrm{Sp}_{> n}$ for any $n\geq 0$ since $\mathrm{Sp}_{> n}$ is closed under colimits, hence  $dF^{*}\Sigma^2TC^+(A)$ is indeed complete. Similarly, define $d\mathcal{Z}_{\bz}^*F^*\Sigma^2TC^+(-)$ by left Kan extension of the $DBF(\bs)$-valued functor $A\mapsto\mathcal{Z}_{\bz}^*F^*\Sigma^2TC^+(A)$. For any quasi-lci ring with bounded torsion $A$, there is a canonical map of bifiltrations $d\mathcal{Z}_{\bz}^*F^{*}\Sigma^2TC^+(A)\rightarrow \mathcal{Z}_{\bz}^*F^{*}\Sigma^2TC^+(A)$, which induces equivalences on graded pieces. Hence it is enough to check that $d\mathcal{Z}_{\bz}^*F^{*}\Sigma^2TC^+(A)$ is bicomplete, which follows from $d\mathcal{Z}_{\bz}^jF^{n}\Sigma^2TC^+(A)\in \mathrm{Sp}_{\geq n+j}$.

\end{proof}

\begin{cor}\label{corfiltS}
For any commutative ring $A$, there is a bicomplete biexhaustive $(\bn^{\mathrm{op}}\times\bn_{>0}^{\mathrm{op}})$-indexed bifiltration $\mathcal{Z}_{\bz}^*F^*\Sigma^2TC^+(A)$ on the spectrum $\Sigma^2TC^+(A)$, extending the bifiltration of Definition \ref{global-bifilt}, with graded pieces
$$\mathrm{gr}_{\mathcal{Z}_{\bz}}^j\mathrm{gr}_{F}^n\Sigma^2TC^+(A)\simeq L\Omega^{<n-j}_{A/\bz}\otimes^L_{\bz}\bz/j\bz[2n-\epsilon_j]. $$
Moreover, the filtration $\mathcal{Z}_{\bz}^*\mathrm{gr}_{F}^n\Sigma^2TC^+(A)$ belongs to $\widehat{DF}(\bz)$ for any $n$. Finally, the functor $A\mapsto \mathcal{Z}_{\bz}^*F^*\Sigma^2TC^+(A)$ satisfies fpqc-descent.
\end{cor}
\begin{proof}
By Corollary \ref{connectivity}, we may extend the functor $\mathcal{Z}_{\bz}^*F^*\Sigma^2TC^+(-)$ to the category of all (simplicial) commutative rings, by left Kan extension of the $DBF(\bs)$-valued functor $A\mapsto\mathcal{Z}_{\bz}^*F^*\Sigma^2TC^+(A)$. The resulting $DBF(\bs)$-valued functor $\mathcal{Z}_{\bz}^*F^*\Sigma^2TC^+(-)$ takes values in $\widehat{DBF}(\bs)$ by Corollary \ref{connectivity} and satisfies fpqc-descent, since its graded pieces do by \cite[Theorem 3.1]{BMS}. 
\end{proof}
\begin{rem}
The complex
$L\Omega_{A/\bs}^{<n}$ is equivalent to a complex of the form
$$[\cdots[[L\Omega^{<n}_{A/\bz}\rightarrow L\Omega^{<n-2}_{A/\bz}/2]\rightarrow  L\Omega^{<n-3}_{A/\bz}/3]\rightarrow \cdots \rightarrow L\Omega^{<1}_{A/\bz}/(n-1)]$$
where we denote by $[C\rightarrow C']$ the fiber of a morphism of complexes $C\rightarrow C'$,  and  $C/j:=C\otimes^L_{\bz}\bz/j$. It would be interesting to understand the maps appearing in this diagram. 
\end{rem}

\begin{rem}
If one defines the spectrum $\Sigma^2TC^+(\X)$ by Zariski descent for any scheme $\X$, then Corollary \ref{corfiltS} extends to arbitrary schemes.
\end{rem}

The proof below relies on Theorem \ref{thmfinal} together with the known computation of $\pi_*THH(\mathcal{O}_F)$.
\begin{ff-proof}
For any $n,j$, we have 
$$\mathrm{gr}^j_{\mathcal{Z}_{\bz}} \mathrm{gr}^{n}_{F}THH(\mathcal{O}_F)\simeq L\Lambda^{n-j}L_{\mathcal{O}_F/\bz}[n-j]\otimes^L_{\bz}\mathrm{gr}^{j}_{F}THH(\bz)$$
by Theorem \ref{thmfinal}. Moreover $L\Lambda^{i}L_{\mathcal{O}_F/\bz}$ is concentrated in homological degree $i-1$ for $i>0$ by \cite[Proposition 5.36]{Flach-Morin18}. 
In particular $\mathrm{gr}^{0}_{F}THH(\mathcal{O}_F)\simeq \mathcal{O}_F$
is concentrated in degree $0$ and
$$\mathrm{gr}^{1}_{F}THH(\mathcal{O}_F)\simeq L_{\mathcal{O}_F/\bz}[1]$$
is concentrated in homological degree $1$. 

If $n\geq 2$ then $\mathrm{gr}^{n}_{F}THH(\mathcal{O}_F)$ is concentrated in homological degrees in $[2n-1,2n-2]$.
Indeed
$$\mathrm{gr}^0_{\mathcal{Z}_{\bz}} \mathrm{gr}^{n}_{F}THH(\mathcal{O}_F)\simeq L\Lambda^{n}L_{\mathcal{O}_F/\bz}[n] $$
is concentrated in degree $2n-1$
and for any $j\geq 1$
$$\mathrm{gr}^j_{\mathcal{Z}_{\bz}} \mathrm{gr}^{n}_{F}THH(\mathcal{O}_F)\simeq L\Lambda^{n-j}L_{\mathcal{O}_F/\bz}[n-j]\otimes^L_{\bz}\bz/j[2j-1]$$
is concentrated in homological degrees in $[2n-1,2n-2]$. We set
$$\mathcal{N}^{n}(\mathcal{O}_F):=\mathrm{gr}^{n}_FTHH(\mathcal{O}_F)[-2n]$$
which we see as a cochain complex.
Hence $\mathcal{N}^{0}(\mathcal{O}_F)$ is concentrated in degree $0$,  $\mathcal{N}^{1}(\mathcal{O}_F)$ is  concentrated in cohomological degree $1$, and $\mathcal{N}^{n}(\mathcal{O}_F)$ is  concentrated in cohomological degrees in $[1,2]$ for any $n\geq 2$. It follows that the spectral sequence
$$E_2^{i,j}=H^{i-j}(\mathcal{N}^{-j}(\mathcal{O}_F))\Rightarrow \pi_{-i-j} THH(\mathcal{O}_F)$$
degenerates and  gives isomorphisms
\begin{eqnarray*}
H^0(\mathcal{N}^{0}(\mathcal{O}_F))&=&\pi_{0} THH(\mathcal{O}_F)= \mathcal{O}_F;\\ 
H^1(\mathcal{N}^{j}(\mathcal{O}_F))&=&\pi_{2j-1} THH(\mathcal{O}_F)= \mathcal{D}^{-1}_F/j \mathcal{O}_F \textrm{ for any }j\geq 1;\\
H^2(\mathcal{N}^{j}(\mathcal{O}_F))&= &\pi_{2j-2} THH(\mathcal{O}_F)=0 \textrm{ for any }j\geq 2;
\end{eqnarray*}
In particular, for any $n\geq 0$, $\mathcal{N}^{n}(\mathcal{O}_F)$ is  concentrated in cohomological degrees in $[0,1]$. Hence
$\mathrm{gr}^{n}_FTHH(\mathcal{O}_F)$ is concentrated in homological degrees $\geq 2n-1$. It follows that $\pi_*(F^n THH(\mathcal{O}_F))$ is concentrated in homological degrees $\geq 2n-1$, hence the map
$$F^n THH(\mathcal{O}_F)\rightarrow THH(\mathcal{O}_F)$$
factors through $\tau_{\geq 2n-1} THH(\mathcal{O}_F)$. This gives a morphism of filtrations
\begin{equation}\label{hi}
F^{*}THH(\mathcal{O}_F)\rightarrow \tau_{\geq 2*-1} THH(\mathcal{O}_F).
\end{equation}
which induces equivalences on graded pieces.

Now we prove the second assertion of the corollary. For any $i\geq 0$, the discussion above shows that the complex $\mathcal{N}^{i}(\mathcal{O}_F)$ is a perfect complex of abelian groups concentrated in cohomological degrees in $[0,1]$. Hence $\mathcal{N}^{i}(\mathcal{O}_F)^{\wedge}_p$ is concentrated in cohomological degrees in $[0,1]$ for any prime number $p$ and any $i\geq 0$. By Propositions \ref{pcomplTHH}, \ref{propmotfiltTP} and \cite[Theorem 1.12]{BMS},  there is a finite decreasing complete exhaustive filtration on $(L\Omega_{\mathcal{O}_F/\bs}^{<n})^{\wedge}_p$, indexed by $i\in[0,n-1]$, with $i$-graded pieces
$\mathcal{N}^{i}(\mathcal{O}_F)^{\wedge}_p$ for $0\leq i<n$. It follows that $(L\Omega_{\mathcal{O}_F/\bs}^{<n})^{\wedge}_p$ is concentrated in cohomological degrees in $[0,1]$ for any prime $p$, hence so is the perfect complex of abelian groups $L\Omega_{\mathcal{O}_F/\bs}^{<n}$.
It follows by induction and completeness of the filtration that $\pi_*(F^n\Sigma^2TC^+(\mathcal{O}_F))$ is concentrated in homological degrees $\geq 2n-1$. We get a morphism of filtrations
\begin{equation}\label{filtmap+}
F^{*}\Sigma^2TC^+(\mathcal{O}_F)\rightarrow \tau_{\geq 2*-1} \Sigma^2TC^+(\mathcal{O}_F).
\end{equation}
The spectral sequence
$$E_2^{i,j}=H^{i-j}(L\Omega_{\mathcal{O}_F/\bs}^{<j})\Rightarrow \pi_{-i-j} \Sigma^2TC^+(\mathcal{O}_F)$$
degenerates and  gives isomorphisms
$$H^0(L\Omega_{\mathcal{O}_F/\bs}^{<n})\simeq \pi_{2n}\Sigma^2TC^+(\mathcal{O}_F)$$
and
$$H^1(L\Omega_{\mathcal{O}_F/\bs}^{<n})\simeq \pi_{2n-1}\Sigma^2TC^+(\mathcal{O}_F)$$
for any $n\geq 0$. It follows that (\ref{filtmap+}) is an equivalence.

\end{ff-proof}

\section{The correction factor $C_{\infty}(\X,n)$}\label{sectionC}
In this section, we prove the statements of Section 1.2. Corollary \ref{Milne} and Theorem \ref{thmintro} follow from Corollary \ref{corCXnintro}, hence it remains to prove Corollary \ref{corCXnintro}. Recall from Definition \ref{defn/S} that we denote
$$R\Gamma(\X,L\Omega_{\X/\bs}^{<n}):=R\Gamma(\X,\mathrm{gr}_{F}^n\Sigma^2TC^+(-))[-2n]$$
for any scheme $\X$. Corollary \ref{corfiltS} immediately gives the following 

\begin{cor}\label{corfilproper}
The complex $R\Gamma(\X,L\Omega_{\X/\bs}^{<n})$ has a complete exhaustive $\bn^{\mathrm{op}}$-indexed filtration $\mathcal{Z}^*$ by $H\bz$-modules with graded pieces
\begin{eqnarray*}
\mathrm{gr}^0_{\mathcal{Z}} R\Gamma(\X,L\Omega_{\X/\bs}^{<n})&\simeq& R\Gamma(\X,L\Omega_{\X/\bz}^{<n});\\
\mathrm{gr}^j_{\mathcal{Z}} R\Gamma(\X,L\Omega_{\X/\bs}^{<n})&\simeq& R\Gamma(\X,L\Omega_{\X/\bz}^{<n-j})\otimes^L_{\bz}\bz/j\bz[-1]
\end{eqnarray*}
for any $j\geq1$.
\end{cor}
If $C$ is a perfect complex of abelian groups with finite cohomology groups we set
$$\chi^{\times}(C):=\prod_{i\in\bz}\mid H^{i}(C) \mid^{(-1)^{i}}.$$

\begin{f-proof}
The scheme $\X$ is regular and proper over $\bz$. It follows that  $R\Gamma(\X,L\Omega_{\X/\bz}^{<n-j})$ is a perfect complex of abelian groups for any $j\geq 0$, see e.g. \cite[Section 5.1]{Flach-Morin18}. In view of Corollary \ref{corfilproper}, the complex $R\Gamma(\X,L\Omega_{\X/\bs}^{<n})$ is also a perfect complex of abelian groups and we have a fiber sequence
$$R\Gamma(\X, \mathcal{Z}^1L\Omega_{\X/\bs}^{<n})\rightarrow R\Gamma(\X,L\Omega_{\X/\bs}^{<n})\rightarrow R\Gamma(\X,L\Omega_{\X/\bz}^{<n})$$
where $R\Gamma(\X, \mathcal{Z}^1L\Omega_{\X/\bs}^{<n})$ has finite cohomology groups. More precisely, $R\Gamma(\X, \mathcal{Z}^1L\Omega_{\X/\bs}^{<n})$ has a finite decreasing complete exhaustive filtration indexed by $j\in[1,n]$ with graded pieces
\begin{equation}\label{grrrr}
\mathrm{gr}^jR\Gamma(\X, \mathcal{Z}^1L\Omega_{\X/\bs}^{<n})\simeq R\Gamma(\X,L\Omega_{\X/\bz}^{<n-j})\otimes^L_{\bz}\bz/j\bz[-1].
\end{equation}
We obtain
\begin{eqnarray}
&&\chi^{\times}\left(R\Gamma(\X, \mathcal{Z}^1L\Omega_{\X/\bs}^{<n})\right)\\
\label{grrrrtooo}&=&\prod_{1\leq j\leq n} \chi^{\times}\left(R\Gamma(\X,L\Omega_{\X/\bz}^{<n-j})\otimes^L_{\bz}\bz/j\bz[-1]\right)\\
\label{Hodgequality}&=&\prod_{1\leq j\leq n} \prod_{0\leq i\leq n-j-1} \chi^{\times}\left(R\Gamma(\X,L\Lambda^{i}L_{\X/\bz})\otimes^L_{\bz}\bz/j\bz[-i-1]\right)\\
\label{mult-add}&=&\prod_{1\leq j\leq n} \prod_{0\leq i\leq n-j-1} j^{\sum_{k\in\bz}(-1)^{k+i+1}\mathrm{dim}_{\bq} H^k(\X_{\bq},\Omega^{i})}\\
&=&\prod_{1\leq j\leq n} \prod_{0\leq i\leq n-j-1}\prod_{k\in \bz} j^{(-1)^{k+i+1}\mathrm{dim}_{\bq} H^k(\X_{\bq},\Omega^{i})}\\
&=&\prod_{0\leq i\leq n-2} \prod_{k\in \bz} \prod_{1\leq j\leq n-i-1} j^{(-1)^{k+i+1}\mathrm{dim}_{\bq} H^k(\X_{\bq},\Omega^{i})}\\
&=&\prod_{0\leq i\leq n-2;k}(n-1-i)!^{(-1)^{k+i+1}\mathrm{dim}_{\bq} H^k(\X_{\bq},\Omega^{i})}\\
&=&C_{\infty}(\X,n)^{-1}.
\end{eqnarray}
Here (\ref{grrrrtooo}) follows from (\ref{grrrr}), (\ref{Hodgequality}) follows from the Hodge filtration and (\ref{mult-add}) is given by Lemma \ref{Lem-multadd} below. Indeed, we have 
$$R\Gamma(\X,L\Lambda^{i}L_{\X/\bz})\otimes_{\bz}\bq\simeq R\Gamma(\X_{\bq},L\Lambda^{i}L_{\X_{\bq}/\bq})\simeq R\Gamma(\X_{\bq},\Omega^{i}_{\X_{\bq}/\bq})$$
since $\X_{\bq}/{\bq}$ is smooth.
\end{f-proof}

\begin{lem}\label{Lem-multadd}
Let $j\geq 1$ and let $C$ be a perfect complex of abelian groups. Then 
$$\chi^{\times}(C\otimes^L_{\bz}\bz/j\bz)= j^{\sum_{i\in\bz} (-1)^{i}\cdot\mathrm{dim}_{\bq}  H^{i}(C_{\bq})}.$$
\end{lem}
\begin{proof}
We may suppose that $C$ is strictly perfect, i.e. that $C^{i}$ is a finitely generated free $\bz$-module for any $i\in\bz$, zero for almost all $i$. Then we have
$$\chi^{\times}(C\otimes^L_{\bz}\bz/j\bz)=j^{\sum_{i\in\bz}(-1)^{i}\cdot \mathrm{rank}_{\bz} C^{i}}= j^{\sum_{i\in\bz} (-1)^{i}\cdot\mathrm{dim}_{\bq}  H^{i}(C_{\bq})}$$
where $C_{\bq}:=C\otimes^L_{\bz}\bq$.
\end{proof}

\begin{rem}
It would be interesting to understand the possible connection between \cite{Hesselholt18} and Conjecture  \ref{conjmain} over finite fields.
\end{rem}

\begin{rem} We expect that Conjecture \ref{conjmain} can be generalized as follows. There should be a Weil-étale cohomology with compact support $R\Gamma_{W,c}(-,\bz(n))$ and a cohomology with compact support $R\Gamma_{c}(-,L\Omega^{<n}_{-/\bs})$ on the category of separated scheme of finite type over $\mathrm{Spec}(\bz)$, with values in perfect complexes of abelian groups. These complexes should recover, for regular schemes proper over $\mathrm{Spec}(\bz)$, the complexes defined in \cite{Flach-Morin18} and in this paper respectively, and should admit fiber sequences for any open-closed decomposition of schemes. We expect the existence of a canonical trivialization
\begin{equation*}
\lambda_\infty:\br\xrightarrow{\sim}\left({\det}_\bz R\Gamma_{W,c}(\mathcal{X},\mathbb{Z}(n))\otimes_\bz {\det}_\bz R\Gamma_c(\X,L\Omega_{\X/\bs}^{<n})\right)\otimes_\bz\br,
\end{equation*}
compatible with the fiber sequences mentioned above, and such that
$$
\lambda_{\infty}(\zeta^*(\mathcal{X},n)^{-1}\cdot\mathbb{Z})= {\det}_\bz R\Gamma_{W,c}(\mathcal{X},\mathbb{Z}(n))\otimes_\bz {\det}_\bz R\Gamma_c(\X,L\Omega_{\X/\bs}^{<n})
$$
for any separated scheme $\X$ of finite type over $\mathrm{Spec}(\bz)$ and any $n\in\bz$.
\end{rem}

\end{document}